\documentclass[1 [leqno,11pt]{amsart}
\usepackage{amssymb, amsmath,amsmath,latexsym,amssymb,amsfonts,amsbsy, amsthm}
\usepackage{graphicx}
\usepackage{pdfsync}

\numberwithin{equation}{section}

\allowdisplaybreaks
\let\al=\alpha

\let\f=\frac

\let\pa=\partial

\def\R{\mathbf R}

\newcommand{\beq}{\begin{equation}}
\newcommand{\eeq}{\end{equation}}
\newcommand{\ben}{\begin{eqnarray}}
\newcommand{\een}{\end{eqnarray}}
\newcommand{\beno}{\begin{eqnarray*}}
\newcommand{\eeno}{\end{eqnarray*}}

\newtheorem{theorem}{Theorem}[section]

\newtheorem{lemma}[theorem]{Lemma}
\newtheorem{proposition}[theorem]{Proposition}

\newtheorem{remark}[theorem]{Remark}

\begin{document}

\title[Asymptotic behavior of the steady Prandtl equation]
{Asymptotic behavior of the steady Prandtl equation}

\author{Yue Wang}
\address{School of Mathematical Sciences, Capital Normal University,100048, Beijing,  China}
\email{yuewang37@pku.edu.cn}

\author{Zhifei Zhang}
\address{School of Mathematical Sciences, Peking University, 100871, Beijing, China}
\email{zfzhang@math.pku.edu.cn}

\begin{abstract}
We study the asymptotic behavior of the Oleinik's solution to the steady Prandtl equation when the outer flow $U(x)=1$. Serrin proved that the Oleinik's solution converges to the famous Blasius solution $\bar u$ in $L^\infty_y$ sense as $x\rightarrow+\infty$. The explicit decay estimates of $u-\bar u$ and its derivatives were proved by Iyer[ARMA 237(2020)] when the initial data is  a small localized perturbation of the Blasius profile. In this paper, we prove the explicit decay estimate of
$\|u(x,y)-\bar{u}(x,y)\|_{L^\infty_y}$ for general initial data with exponential decay.  We also prove the decay estimates of its derivatives when the data has an additional concave assumption. Our proof is based on the maximum principle technique. The key ingredient is to find a series of barrier functions.
\end{abstract}

\date{\today}
\maketitle

\section{Introduction}
We study the steady Prandtl equation
\begin{equation}\label{eq:SP}
  \left\{
  \begin{aligned}
    &u\pa_x u +v\pa_{y}u-\pa_{y}^2u=\f {dp} {dx},\quad (x, y)\in \R_+\times \R_+,\\
    &\pa_xu+\pa_y v=0,\\&u|_{x=0}=u_0(y),\quad u|_{y=0}=v|_{y=0}=0,\\
    & \displaystyle\lim_{y\to+\infty} u(x,y)=U(x).
    \end{aligned}
  \right.
\end{equation}
Here $(u,v)$ is the velocity and the outer flow $(U(x),p(x))$ satisfies the Bernoulli's law:
\beno
U(x)U'(x)+p'(x)=0.
\eeno
In this paper, we consider the case when the gradient of pressure $p'(x)=0$. Thus, the outer flow $U(x)$ is a constant in this case. For simplicity, we take $U\equiv1.$

Let us introduce the Von Mises transformation $(x,\psi)$ defined by
\begin{align}\label{psi1}
x=x,\,\, \psi=\psi(x,y)=\int_0^y u(x,y')dy'.
\end{align}
Introduce the new unknown $w(x,\psi)=u(x,y)^2.$  A direct calculation shows that
  \begin{align}\label{PvM}
  \begin{split}
& 2\partial_y u =\partial_\psi w,\quad2\partial^2_{y} u=\sqrt{w}\partial^2_{\psi}w.
 \end{split}
\end{align}
Hence, the Prandtl system is reduced to a parabolic type equation(view $x$ as time direction):
\begin{align}\label{starmain}
\partial_x w=\sqrt{w}\partial_{\psi}^2 w
\end{align}
along with the boundary conditions
\begin{align}\label{IBCW}
\begin{split}
&w(x,0)=0,\quad w(0,\psi)=w_0(\psi),\\
&w(x,\psi)\rightarrow
1\quad \text{as}\,\,\psi\rightarrow +\infty.
\end{split}
\end{align}

Based on the Von Mises transformation and maximum principle technique,  Oleinik proved the following existence and uniqueness result of classical solution(see Theorem 2.1.1 in \cite{Olei}).

\begin{theorem}\label{thm:olei}(Oleinik)
If $\frac{dp}{dx}=0$ and the initial data $u_0$ satisfies
\begin{align}\label{OI}\begin{split}
 &u_0(y)\in C_b^{2,\alpha}\big([0,+\infty)\big)(\al>0),\quad u_0(0)=0,\,\,u'_0(0)>0,\\
 &u_0(y)>0\,\, \text{for}\, \,y\in(0,+\infty),\,\, u''_0(y)=O(y^2),\end{split}
\end{align} then the steady Prandtl equation \eqref{eq:SP} admits a global-in-$x$ solution $u_0\in C^1(\R_+\times\R_+)$ with the following properties: for any $X>0,$

1. $u$ is bounded and continuous in $[0,X]\times \R_+;$

2. $u(x,y)>0$ for $y>0;$

3. $u_{y},u_{yy}$ are bounded and continuous in $[0,X]\times\R_+$;

4. $v,v_{y},u_{x}$ are locally bounded and continuous in $[0,X]\times\R_+.$
\end{theorem}

In fact, Theorem 2.1.1 in \cite{Olei} showed that in the case of favorable pressure gradient $p'(x)\leq 0$, the global-in-$x$ solutions exist; while in the case of adverse pressure gradient $p'(x)> 0,$ only local-in-$x$ solutions exist.

Recently, in the case of favorable pressure gradient, Guo and Iyer \cite{Guo2} proved the higher regularity of the solution through the energy method, and the authors proved the global-in-$x$ $C^\infty$ regularity up to the boundary $y=0$ using the maximum principle technique \cite{WZ19}. In the case of adverse pressure gradient, Dalibard and  Masmoudi \cite{DM} as well as Shen and the authors \cite{SWZ} justified the physical phenomenon of boundary layer separation.
Let us mention some related works \cite{XZ, EE, KVW} for the unsteady Prandtl equation.

In this paper, we are concerned with the asymptotic behavior of Oleinik's solution when the outer flow $U(x)=1$.
In this case, \eqref{eq:SP} admits a family of self-similar Blasius solutions:
\begin{align}\label{zetadef}
 [\bar{u},\bar{v}]=\Big[f'(\zeta),\frac{1}{2\sqrt{x+x_0}}\{\zeta f'(\zeta)-f(\zeta)\}\Big],
\end{align}
where $\zeta=\frac{ y}{\sqrt{x+x_0}}$ with $x_0>0$ as a free parameter. See section 2.1 for the properties of $f(\zeta)$.
For simplicity, we always take $x_0=1.$ \smallskip

Serrin \cite{Serrin}  proved the following asymptotic behavior of the Oeinik's solution.

\begin{theorem}\label{Sethm}(Serrin) Let $u$ be a global Oleinik's solution to \eqref{eq:SP} with $U(x)=1$.
Then the asymptotic behavior holds
$$\|u(x,y)-\bar{u}(x,y)\|_{L^\infty_y}\rightarrow 0\quad \text{as}\quad x\rightarrow+\infty.$$
\end{theorem}

Recently, Iyer \cite{Iyer} proved the explicit decay estimates of modulated substraction $\phi$ and its derivatives in the Von Mises coordinates when the initial data is a small localized perturbation of the Blasius profile by using the energy method, where
\begin{align}\label{modsubstra}\begin{split}
    \phi(x,\psi)&=w(x,\psi)-\bar{w}(x,\psi),\quad \bar w=\bar u^2.
    \end{split}
\end{align}
These estimates play a crucial role in validating the Prandtl's boundary layer theory \cite{Guo1, IM}.\smallskip

The aim of this paper is to study the asymptotic behavior of the Oleinik's solution for general initial data.
The first main result is stated as follows.

\begin{theorem}\label{thm:decay0}
Let $u$ be a global Oleinik's solution to \eqref{eq:SP} with $U(x)=1$.  Under the additional decay assumption
 \begin{align}\label{dk0}
 |u_0(y)-1|\leq C_4e^{-y^2C_5}
\end{align}
for some positive constants $C_4,C_5$ with $C_5>C_1$ where $C_1$ is the constant in \eqref{blainf},  there exist positive constants $C $ and $c$ so that  for any $(x,y)\in \R_+\times \R_+$,
$$|u(x,y)-\bar{u}(x,y)|\leq \frac{C}{\sqrt{x+1}}\ln (x+e)e^{-c\frac{y^2}{x+1} }.$$
\end{theorem}

\begin{remark}
The decay rate should be optimal in the sense explained on Page 6 in \cite{IM}.
\end{remark}

For the decay estimates of high order derivatives of $u$,  we additionally require the initial data to be  concave.

 \begin{theorem}\label{thm:decay1}
 Let $u$ be a global Oleinik's solution to \eqref{eq:SP} with $U(x)=1$.  Under the assumptions \eqref{dk0}
 and
 \begin{align}\label{decay2inf}
 C_6e^{-y^2C_7}\leq \partial_y^2u_0(y)\leq 0
\end{align}
for some positive constants $C_6, C_7$ with $C_7>C_1$ where $C_1$ is the constant in \eqref{blainf},
there exist positive constants $C,c$ and $N$ so that for any $(x,y)\in (N,+\infty)\times \R_+$,
\begin{align*}
 & -\frac{C}{x+1}e^{-c\frac{y^2}{x+1}}\leq \partial_y^2u(x,y)\leq 0,\\
  & |\partial_y (u(x,y)-\bar{u}(x,y))|\leq \frac{C}{(x+1)^{\frac{3}{4}}}\ln(x+e)e^{-c\frac{y^2}{x+1} },\\
  & |\partial_xu(x,y)|\leq \frac{C}{x+1}e^{-c\frac{y^2}{x+1}},\quad  |\partial_{xy} u(x,y)|\leq \frac{C}{(x+1)^{\frac{3}{4}}}e^{-c\frac{y^2}{x+1}}.
  \end{align*}
\end{theorem}

\begin{remark}
These decay estimates mean that $u$ has similar behaviors  with the Blasius solution in the large time. It remains unknown
whether the concave condition could be removed.
\end{remark}

The proof of Theorem \ref{thm:decay0} and Theorem \ref{thm:decay1} is based on decay estimates for $\phi$ and $w$ under the
Von Mises coordinates.

\begin{theorem}\label{thm:decay0-V}
Under the assumptions in Theorem \ref{thm:decay1}, there exist positive constants $C$ and $c$ so that  for any $(x,\psi)\in \R_+\times \R_+$,
\begin{align*}
&|\phi(x,\psi)|\leq \frac{C }{\sqrt{x+1}}e^{-c\frac{\psi^2}{x+1}},\quad
 |\partial_\psi^2 \phi(x,\psi)|\leq\frac{C}{x+1}e^{-c\frac{\psi^2}{x+1}},\\
&|\partial_{x}\phi(x,\psi)|\leq \frac{C }{x+1}e^{-c\frac{\psi^2}{x+1}},\quad
|\partial_\psi \phi(x,\psi) |\leq \frac{C}{(x+1)^{\frac{3}{4}}}e^{-c\frac{\psi^2}{x+1} }.
\end{align*}
\end{theorem}

\begin{theorem}\label{thm:w-decay}
Under the assumptions in Theorem \ref{thm:decay1}, there exist positive constants $C, c$ and $N$ so that for any $(x,\psi)\in (N,+\infty)\times \R_+$,
\begin{align*}
&|\partial_{\psi x} w(x,\psi)| \leq \frac{C}{(x+1)^{\frac{3}{4}}}e^{-c\frac{\psi^2}{x+1} },\\
& |\partial_x^2 w(x,\psi)|+|\partial_x\partial_\psi^2 w(x,\psi)|\leq \frac{C}{(x+1)^{\frac{1}{2}}}e^{-c\frac{\psi^2}{x+1} }.
\end{align*}
\end{theorem}

The proof of Theorem \ref{thm:decay0-V} and Theorem \ref{thm:w-decay} used the maximum principle technique. The key ingredient is to find a series of barrier functions (with ridges), whose constructions depend on the structure of Blasius profile. In fact, we provide the pointwise estimates including the decay rate with respect to $\psi$ near $\psi=0,$ which is crucial when we derive the decay estimates under the Euler coordinates from the results obtained under the Von Mises coordinates.

\section{Blasius profile and Von Mises coordinates}\label{Prl}

\subsection{Blasius profile}

The Blasius profile $f(\zeta)$ satisfies
\begin{align}\label{BE}\begin{split}
&\frac{1}{2}ff''+f'''=0,\,\,f(0)=f'(0)=0,
\\&f'(\zeta)\rightarrow1\quad \text{and}\quad \frac{f(\zeta)}{\zeta}\rightarrow1\quad \,\text{as}\,\,\zeta\rightarrow+\infty,
\\&0\leq f'(\zeta)\leq 1\quad \text{and}\quad f''(\zeta)\geq 0\quad \text{for}\,\,\zeta\geq 0\\
&0< f''(0)=b_0,\quad f'''(\zeta)<0\quad \text{for}\,\,\zeta>0.
\end{split}
\end{align}
There exist positive constants $C_1,C_2$ so that
\begin{align}\label{blainf}
   1-f'(\zeta)\sim \zeta^{-1}e^{-\zeta^2C_1-C_2\zeta},\quad f''(\zeta)\sim \zeta(1-f')\sim e^{-\zeta^2C_1-C_2\zeta},
\end{align}
as $\zeta\rightarrow+\infty.$

\begin{lemma}\label{lem:f}
It holds that
\beno
f^{(3)}(0)=0,\quad f^{(4)}(0)=0,\quad f^{(5)}(0)<0.
\eeno
\end{lemma}
\begin{proof}
By $\frac{1}{2}ff''+f'''=0,$ we have $f^{(3)}(0)=0$ and
\begin{align}\label{star}\begin{split}
&\frac{1}{2}f'f''+\frac{1}{2}ff^{(3)}+f^{(4)}=0,\\
&\frac{1}{2}(f'')^2+\frac{1}{2}f'f^{(3)}+\frac{1}{2}ff^{(4)}+\frac{1}{2}f'f^{(3)}+f^{(5)}=0.
\end{split}\end{align}
By \eqref{BE} and evaluating at $\zeta=0,$ the result follows.
\end{proof}

\subsection{A comparison lemma}

\begin{lemma}\label{lem:com}
There exist positive constants $c<1$ and $C>1$ depending on $w_0$ such that
 \begin{align}\label{star1}
 c\bar{w}\leq w\leq C\bar{w}\quad \text{in}\quad \R_+\times \R_+.
\end{align}
\end{lemma}

\begin{proof}
Since $u_0(0)=0$ and  $u_0'(0)>0,$ by \eqref{PvM},
$$w_0(0)=0,\quad\partial_\psi w(0,\psi)\sim 1\quad \psi \,\, \text{near} \,\, 0,$$
which gives
\begin{align}\nonumber
 \bar{w}(0,\psi) \sim\psi\sim w(0,\psi)\quad \psi\,\,\text{near}\,\,0.
\end{align}
Thanks to $u_0(y)>0$ and $\bar{u}(0,y)>0$ for $y>0$, we have $\bar{w}>0$ and $w>0$ for $\psi>0.$
Moreover, $w(x,\psi),\,\,\bar{w}(x,\psi)\rightarrow 1$  as $\psi\rightarrow+\infty$.
Hence, away from $0$, both $w_0(\psi)$ and $\bar{w}(0,\psi)$ have positive minimum and maximum.
Then there exist some positive constants $c<1$ and $1<C$ so that
\begin{align}\label{nzic}
 cw_0(\psi)\leq \bar{w}(0,\psi) \leq &Cw_0(\psi).
\end{align}

Take $c$ and $C$ to be the constants in \eqref{nzic}. For any positive constant $b,$ we have
$$\partial_x(w-b\bar{w})-\sqrt{w}\partial_\psi^2(w-b\bar{w})=(\sqrt{b}\sqrt{w}-b\sqrt{\bar{w}})\partial_\psi^2 \bar{w}+(b-\sqrt{b})\sqrt{w}\partial_\psi^2 \bar{w},$$
where we note $\partial_\psi^2 \bar{w}<0$ for $\psi>0.$

 We  first prove that $w-C\bar{w}\leq 0.$ Otherwise, since $w-C\bar{w}=0$ on $\psi=0$,  \eqref{star1} holds on $x=0$ due to \eqref{nzic}, and $w-C\bar{w}\rightarrow1-C<0$ as $\psi\rightarrow+\infty$, a positive maximum is obtained at some point $(x_0,\psi_0)\in(0,x_0]\times (0,+\infty)$ with $(w-C\bar{w})(x_0,\psi_0)>0$. This implies
 $$(\sqrt{w}-\sqrt{C}\sqrt{\bar{w}})(x_0,\psi_0)>0.$$
On the other hand, at $(x_0,\psi_0),$
$$\partial_x(w-C\bar{w})-\sqrt{w}\partial_\psi^2(w-C\bar{w})=\sqrt{C}(\sqrt{w}-\sqrt{C}\sqrt{\bar{w}})\partial_\psi^2 \bar{w}+(C-\sqrt{C})\sqrt{w}\partial_\psi^2 \bar{w}<0,$$
which contradicts to the property of maximum point. Hence, $w-C\bar{w}\leq 0.$

The proof of  $w-c\bar{w}\geq0$ is similar. At the negative minimum point, $\sqrt{w}<\sqrt{c}\sqrt{\bar{w}}$ and there holds
$$
\partial_x(w-c\bar{w})-\sqrt{w}\partial_\psi^2(w-c\bar{w})=(\sqrt{c}\sqrt{w}-c\sqrt{\bar{w}})\partial_\psi^2 \bar{w}+(c-\sqrt{c})\sqrt{w}\partial_\psi^2 \bar{w}>0,
$$
which also leads to a contradiction.
\end{proof}

\subsection{Von Mises coordinates} By \eqref{psi1}, we introduce the notation
\begin{align}\label{y}
   y(\psi;u)=\int_0^\psi\frac{1}{\sqrt{w}(x,\psi')} d\psi'.
\end{align}
In particular,  $y(\psi;\bar{u})=\int_0^\psi\frac{1}{\sqrt{\bar{w}}(x,\psi')} d\psi',$ corresponds to the Blasius profile.
It follows from Lemma \ref{lem:com} that there exist positive constants $c$ and $C$ such that
\begin{align}\label{yyyorder}
cy(\psi;\bar{u})\leq y(\psi;u)\leq Cy(\psi;\bar{u}).
\end{align}

In what follows, we always denote
\ben\label{def:h-zeta}
 h=\frac{\psi}{\sqrt{x+1}},\quad \zeta=\frac{y(\psi;\bar{u})}{\sqrt{x+1}}.
 \een

We infer from \eqref{psi1} that for $y=y(\psi;\bar u)$,
\begin{align}\nonumber
   \psi=\int_0^{y} \bar{u}(x,y')dy'=\sqrt{x+1}\int_0^{\zeta} f'(\zeta)d\zeta=\sqrt{x+1}f(\zeta),
\end{align}
which gives
\begin{align}\label{psify}
   h=\frac{\psi}{\sqrt{x+1}}=f(\zeta).
\end{align}
Since $f''\geq 0,\,f''(0)>0$ and $f'(0)=0$, it holds that $f'(\zeta)>0$ for $\zeta>0,$ and thus $f$ is strictly increasing.
Hence, $\zeta\stackrel{one \, to\, one}{\longleftrightarrow}h$.  By \eqref{BE}, there exists a large positive constant $M$ such that
\begin{align}\label{zeps}
\frac{1}{2}\zeta\leq \frac{\psi}{\sqrt{x+1}}=f(\zeta)\leq 2\zeta
\end{align}
when $\zeta\ge M$ or $h\ge M$.

Since $f'(\zeta)\sim \zeta,\,\,f(\zeta)\sim \zeta^2$ for $\zeta$  near $0$ due to \eqref{BE}, it holds that for any $a>0$,
 \begin{align}\label{wh}
  c_a h\leq \bar{w}(x,\psi)\leq C_a h\quad\text{for}\,\, h\leq a.
 \end{align}

 Recall $\bar w=\bar u^2=f'(\zeta)^2$. By \eqref{PvM} and \eqref{starmain}, we have the following relations which will be frequently used:
\begin{align}\label{fu}
\begin{split}
-\partial_{x}\bar{w}=\frac{1}{x+1}ff'',\quad \partial_\psi \bar{w}= \frac{2}{\sqrt{x+1}}f''
.\end{split}
\end{align}
From \eqref{fu}, \eqref{BE} and \eqref{blainf}, it holds  hat
\begin{align}\label{xw}\begin{split}
   -\partial_{x}\bar{w}&\geq c\frac{1}{x+1}\zeta^2\quad\text{for}\quad\zeta\leq 1,\\
   -\partial_{x}\bar{w}&\geq c\frac{1}{x+1}\zeta f''\quad\text{for}\quad\zeta> 1.
   \end{split}
\end{align}

From Lemma \ref{lem:f}, \eqref{BE} and \eqref{fu}, it is easy to see that

\begin{lemma}\label{771} For any fixed $x\in[0,+\infty),$ $-\partial_{x}\bar{w}$ is increasing with respect to $\psi$ and
$-\partial_{x}\bar{w}$ is positive for $\psi>0$.
\end{lemma}

%\begin{proof}
%Thanks to
%\begin{align}
% \partial_{x}\bar{w}=\sqrt{\bar{w}}\partial_{\psi}^2 \bar{w}=2\partial_{y}^2\bar{u}|_{(x,y(\psi;\bar{u}))}=2\frac{1}{x+1}f'''|_{\zeta=\frac{y(\psi;\bar{u})}{\sqrt{x+1}}},
%\end{align}
% we infer that for $\psi>0$
%\begin{align}\label{wxf}
% -\partial_{x}\bar{w}=\frac{1}{x+1}ff''|_{\zeta=\frac{y(\psi;\bar{u})}{\sqrt{x+1}}}>0.
%\end{align}
%So, $-\partial_{x}\bar{w}$ is increasing with respect to $\psi$. Due to $f'''(\zeta)<0$ for $\zeta>0$,  $-\partial_{x}\bar{w}$ is positive for $\psi>0$.
%\end{proof}

If we use $(\tilde{x},\psi)$ to denote the Von Mises variables to avoid
confusion for a while, then it holds that
\ben
\partial_{\tilde{x}}= \partial_x-\frac{ \int_0^{y}\bar{u}_x(x,y')dy'}{\bar{u}}\partial_{y}.\label{eq:EV}
\een

\subsection{Some properties of $w$}\label{Pw}
From Theorem 2.1.14, Lemma 2.1.9 and Lemma 2.1.12 in \cite{Olei} and Lemma 3.1 in \cite{DM}, we know that

\begin{itemize}

\item[1.] $\partial_x w(x,0)=0.$

\item[2.] $\displaystyle\lim_{\psi\to+\infty} \partial^2_\psi w(x,\psi)= 0,$ which implies
 \begin{align}\label{wxinf0}
 \displaystyle\lim_{\psi\to+\infty} \partial_x w(x,\psi)= 0.
 \end{align}

\item[3.]  For any $\bar{x}>0,$ there exist $y_0>0,m>0$ such that
$$
\partial_yu(x,y)\geq m\quad \text{in}\,\, [0,\bar{x}]\times[0,y_0].
$$

\item[4.]  For any $\bar{x}>0,$ there exist positive constants $\psi_1$ and $M$  such that
\begin{align}\label{owp}
|\partial_x w|\leq M\psi^{1-\beta}\quad \text{in}\,\,[0,\bar{x}]\times[0,\psi_1],
\end{align}
where $\beta\in(0,\frac{1}{2})$ (see page 25 in \cite{Olei}).
\end{itemize}

%Moreover,  by Lemma \ref{lem:com}, we have
%\begin{align}\label{wlbd}
%C\geq w(x,\psi)\geq c\bar{w}(x,\psi)=cf'\Big(\frac{y(\psi;\bar{u})}{\sqrt{x+1}}\Big)^2.
%\end{align}

\section{Convergence to the Blasius solution}\label{usrd}

In this section, we prove Theorem \ref{thm:decay0}. Throughout this section, we assume that $u$ is an Oleinik's solution with the initial data satisfying \eqref{dk0}.  We denote
$$Lv=\partial_x v-\sqrt{w}\partial_{\psi}^2 v.$$

\subsection{The perturbation equation}
We denote
\beno
 \phi(x,\psi)=w(x,\psi)-\bar{w}(x,\psi).
 \eeno
 A straight calculation gives
\begin{align}\label{Main}\begin{split}
\partial_x& \phi-\sqrt{w}\partial_{\psi}^2 \phi+A\phi=0, \\ A(x,\psi)&=-\frac{\partial_\psi^2 \bar{w}}{\sqrt{\bar{w}}+\sqrt{w}}(x,\psi)=-\frac{\partial_x \bar{w}}{\sqrt{\bar{w}}(\sqrt{\bar{w}}+\sqrt{w})}(x,\psi)\\&=-2\frac{\bar{u}_{yy}|_{(x,y)=(x,y(\psi;\bar{u}))}}{\bar{u}|_{(x,y)=(x,y(\psi;\bar{u}))}\big(\bar{u}|_{(x,y)=(x,y(\psi;\bar{u}))}+u|_{(x,y)=(x,y(\psi;u))})}.
\end{split}
\end{align}

Let us derive some useful properties of $A$.

\begin{lemma}\label{lem:A}
It holds that for any $(x,\psi)\in \R_+\times \R_+$,
$$|A(x,\psi)|\leq\frac{C}{x+1},$$
and for any $k_0\in(0,+\infty),$ there exists a positive constant $\lambda_{k_0}$ such that
$$A(x,\psi)>\frac{\lambda_{k_0}}{x+1}\quad\text{for}\quad \zeta=\frac{y(\psi;\bar{u})}{\sqrt{x+1}}\leq k_0,$$
which implies that $A>0$.
\end{lemma}

\begin{proof}
By Lemma \ref{lem:com}, we have
$$u(x,y(\psi;u))=\sqrt{w}(x,\psi)\sim \sqrt{\bar{w}}(x,\psi)=\bar{u}(x,y(\psi;\bar{u})).$$
Then we infer from \eqref{Main} that
  \begin{align}\label{Ah}
  \frac{c}{x+1}\frac{-f'''(\zeta)}{(f'(\zeta))^2}\leq A\leq\frac{C}{x+1}\frac{-f'''(\zeta)}{(f'(\zeta))^2},
  \end{align}
  where $\zeta=\frac{y(\psi;\bar{u})}{\sqrt{x+1}}.$ Due to \eqref{BE}, we have
$$\frac{-f'''(\zeta)}{(f'(\zeta))^2}=\frac{\frac{1}{2}ff''(\zeta)}{(f'(\zeta))^2}=\frac{1}{4}\quad \text{as}\,\,\zeta\rightarrow0.$$
Then our result follows from \eqref{BE} and \eqref{blainf}.
\end{proof}

\begin{remark}
Since $A>0$, the term $A\phi$ could be viewed as a damping term. Then it is natural to expect that $\phi$ will converge to
zero in the large time.
\end{remark}

\subsection{Preliminary decay estimates}

 \begin{lemma}\label{lem:exp}
 There exist a large positive constant $C$ and a small positive constant $\varepsilon$ such that
 $$|\phi(x,\psi)|\leq Ce^{-\frac{\psi^2}{x+1}\varepsilon}\quad \text{in}\quad \R_+\times \R_+.$$
\end{lemma}

\begin{proof}
Thanks to
 $$\psi(0,y)=\int_0^y u(0,y')dy'=\int_0^y u_0(y')dy'$$
 and $u_0(y)\rightarrow\,1$ as $y\rightarrow+\infty$, there exists a large positive constant $ N$ such that at $x=0$,
 \begin{align}\label{epyypsi}
 \f12y(\psi;u)\leq \psi\leq 2y(\psi;u)\quad \text{for}\,\,\psi>N.
\end{align}
 Hence, by \eqref{dk0} and \eqref{epyypsi},  we get
 \begin{align}\label{3.2}
|\sqrt{w_0}(\psi)-1|\le C\partial_\psi \bar{w}(0,\psi)\quad\text{for}\,\,\psi>N,
\end{align}
where we used the fact that
\begin{align}\label{psiwyb}
\begin{split}
 \partial_\psi \bar{w}(x,\psi)&=2\partial_{y} \bar{u} (x,y(\psi;\bar{u}))= \frac{2}{\sqrt{x+1}}f''\Big(\frac{y(\psi;\bar{u})}{\sqrt{x+1}}\Big),\\ \partial_\psi \bar{w}&\sim C_0\frac{2}{\sqrt{x+1}}e^{-\zeta^2C_1-C_2\zeta}\quad\text{as}\,\,\zeta\rightarrow+\infty.
\end{split}
\end{align}
On the other hand, by \eqref{BE}, for $\psi>N$,
\begin{align*}
   0\leq 1-\sqrt{\bar{w}}(0,\psi)=1-f'(\zeta)
    \leq Cf''(\zeta)=\frac{C}{2}\partial_\psi \bar{w}(0,\psi).
\end{align*}
For $\psi\leq N,$ there exists a positive constant $a_0$ such that $\partial_\psi \bar{w}(0,\psi)>a_0$. This along with \eqref{3.2} ensures that for $\psi\geq0,$
\begin{align}\label{infic}
|w_0(\psi)-\bar{w}(0,\psi)|\le C\partial_\psi \bar{w}(0,\psi).
\end{align}

Now we claim that $0\leq Ce^{-\frac{\psi^2}{x+1}\varepsilon}\pm\phi.$ Otherwise, since
$$\phi(x,0)=0,\quad \phi\rightarrow0\,\,\text{as}\,\,\,\psi\rightarrow\infty,$$
and $|\phi(0,\psi)|<Ce^{-\psi^2\varepsilon}$ for a small positive $\varepsilon$ due to \eqref{psiwyb} and \eqref{infic}, a negative minimum is obtained at some point $(x_0,\psi_0)\in(0,x_0]\times (0,+\infty)$ with $\big(Ce^{-\frac{\psi^2}{x+1}\varepsilon}\pm\phi\big)(x_0,\psi_0)<0$.

On the other hand, $L(\pm\phi)+A(\pm\phi)=0$ and
 $$Le^{-\frac{\psi^2}{x+1}\varepsilon}=e^{-\frac{\psi^2}{x+1}\varepsilon}\Big[\varepsilon\frac{\psi^2}{(x+1)^2}
 -\sqrt{w}\big(-\varepsilon\frac{2}{x+1}+\varepsilon^2\frac{4\psi^2}{(x+1)^2}\big)\Big]>0,$$
 by Lemma \ref{lem:com} and taking $\varepsilon$ small enough. Therefore, at $(x_0,\psi_0),$
 \beno
 L(Ce^{-\frac{\psi^2}{x+1}\varepsilon})+A(Ce^{-\frac{\psi^2}{x+1}\varepsilon})+L(\pm\phi)+A(\pm\phi)>0,
 \eeno
 due to $A\geq0.$  However, by the property of negative minimum point, at $(x_0,\psi_0),$
 \beno
 L\big(Ce^{-\frac{\psi^2}{x+1}\varepsilon}\pm\phi\big)+A\big(Ce^{-\frac{\psi^2}{x+1}\varepsilon}\pm\phi\big)\le 0,
 \eeno
  which is a contradiction.
  \end{proof}

  \begin{lemma}\label{lem:alg}
 There exists a positive constant $C$ and a small positive constant $\lambda$ such that $$|\phi(x,\psi)|<C(x+1)^{-\lambda}\quad \text{in}\quad \R_+\times \R_+.$$

  \end{lemma}

  \begin{proof}

Let
\beno
 g(x,\psi)=C(x+1)^{-\lambda}\left\{
\begin{aligned}
h^{\frac{1}{2}}M^{\frac{1}{2}},\quad&h\leq \frac{1}{M},\\
1,\quad&\quad\frac{1}{M}\le h\le h_0,\\
\frac{1}{h^{2+2\lambda}}h_0^{2+2\lambda},\quad&h\geq h_0,
\end{aligned}
\right.
\eeno
where $\lambda\in(0,1)$ is a positive constant to be determined.

First of all, it holds that
\beno
(g\pm\phi)(x,0)= 0,\quad g\pm\phi\rightarrow 0\quad \text{as}\,\,h\rightarrow +\infty,
\eeno
and by \eqref{infic} and \eqref{wh},
\beno
g\pm\phi\geq 0\quad \text{on}\,\,\,x=0
\eeno
by taking $C$ large.

Now we claim that $g\pm\phi\geq 0.$ Otherwise, by the initial and boundary conditions, a negative minimum is obtained at some point $(x_0,\psi_0)\in(0,x_0]\times (0,+\infty)$ with $(g\pm\phi)(x_0,\psi_0)<0$. In the following, we work in the domain $(0,x_0]\times (0,+\infty).$

%Note that in $[0,x_0]\times [0,+\infty)$, $g\pm\phi\rightarrow 0$
%as $\psi\rightarrow +\infty.$

In $\{h> h_0\},$ by Lemma \ref{lem:com}, we have
\beno
L\frac{(x+1)^{-\lambda+1+\lambda}}{\psi^{2+2\lambda}}
=\frac{1}{\psi^{2+2\lambda}}\Big(1-\sqrt{w}(2+2\lambda)(3+2\lambda)\frac{x+1}{\psi^2})\Big)>0
\eeno
 by taking $h_0$ large independent of $\lambda\in(0,1)$. Hence, the minimum cannot be achieved in $\{h> h_0\}.$

By Lemma \ref{lem:A}, there exists a positive constant $\lambda_0$ such that
\begin{align}\nonumber
A>\frac{\lambda_0}{x+1}\quad\text{for}\,\,h\leq h_0.
\end{align}
Now we take $\lambda =\lambda_0.$ In $\{\frac{1}{M}< h<h_0\},$
\begin{align}\nonumber
    L(x+1)^{-\lambda}+A(x+1)^{-\lambda}\geq-\lambda(x+1)^{-\lambda-1}+A(x+1)^{-\lambda}>0.
\end{align} Hence, the minimum cannot be achieved in $\{\frac{1}{M}< h< h_0\}.$

In $\{h<\frac{1}{M}\},$ by Lemma \ref{lem:com} and \eqref{psify},
\begin{align}\nonumber
 L\big[(x+1)^{-\lambda-\frac{1}{4}}\psi^{\frac{1}{2}}\big]=(x+1)^{-\lambda-\frac{1}{4}-1}\psi^{\frac{1}{2}}
\Big(-\lambda-\frac{1}{4}+\frac{1}{4}\sqrt{w}\frac{x+1}{\psi^2}\Big)>0
\end{align}
by taking $M$ large. Indeed, by Lemma \ref{lem:com}, it holds
$$\sqrt{w}(x,\psi)\geq c \sqrt{\bar{w}}(x,\psi),$$
which along with \eqref{wh} gives
$$\sqrt{w}(x,\psi)\geq c\frac{\psi^{\frac{1}{2}}}{(x+1)^{\frac{1}{4}}},\quad h\leq 1.$$
Therefore, for $M$ large, the minimum cannot be achieved in $\{h<\frac{1}{M}\}.$

On the other hand, the minimum cannot be achieved at lines $h=\frac{1}{M}$ and $h=h_0$ since they are ridges with respect to $h$ for any fixed $x$. In summary, there is no negative minimum point $(x_0,\psi_0)$ in the interior.
\end{proof}

\begin{remark} Serrin \cite{Serrin} constructed barrier functions with ridges and for readers' convenience, here is a brief description. If $\varphi(s)\in C^1((a,c)\cup(c,b))\cap C((a,b))$ and $\varphi'_{-}(c)>\varphi'_{+}(c)$, then we call $``x=c"$ a ridge. The following figures are four examples.
\begin{figure}[h]
\centering
\includegraphics{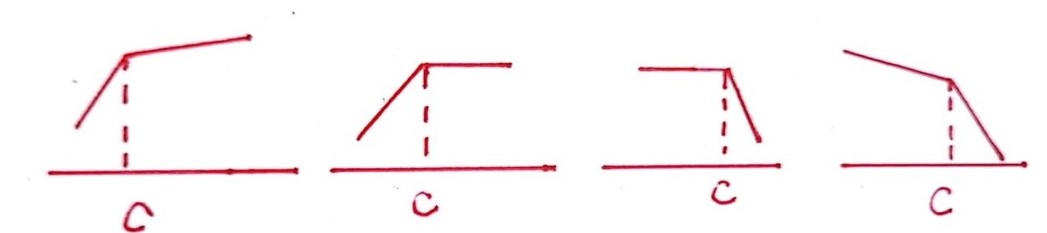}
\caption{Ridges}
\end{figure}
\end{remark}

\subsection{Decay estimates under Von Mises coordinates}\label{0udk1}

\begin{proposition}\label{51}
For any fixed $\alpha\in(0,1),$ there exist positive constants $C_B$, $B$ and $N$ large and a small positive constant $\lambda>0$ such that
\beno
|\phi(x,\psi)|\leq g(x,\psi)e^{-B(x+1)^{-\frac{\lambda}{2}}} \quad\text{in}\quad \R_+\times \R_+,
\eeno
 where
 \beno
g(x,\psi)=C_B\left\{\begin{aligned}N^{1-\alpha}(x+1)^{-\frac{1}{2}-\frac{1-\alpha}{2}}\psi^{1-\alpha},&\quad h<\frac{1}{N},\\
\frac{1}{b_0}\partial_\psi \bar{w},&\quad h\geq \frac{1}{N},\end{aligned}\right.
\eeno
 with $b_0=2f''(\zeta_0)$ and  $f(\zeta_0)=\frac{1}{N}.$
 \end{proposition}

\begin{remark}
At $h=\frac{1}{N}$, $\partial_\psi \bar{w}(x,\psi)= \frac{2}{\sqrt{x+1}}f''(\zeta_0).$
Hence,  $g$ is continuous at $h=\frac{1}{N}$.
\end{remark}

\begin{proof}
Note $e^{-B}\leq e^{-B(x+1)^{-\frac{\lambda}{2}}}\leq 1.$ Take $C_{B}\in(e^B,+\infty)$ to be a large constant such that by \eqref{infic}, \eqref{nzic} and \eqref{wh}, $$(ge^{-B(x+1)^{-\frac{\lambda}{2}}}\pm\phi)(0,\psi) \geq 0.$$ On the other hand, by \eqref{IBCW}, \eqref{psiwyb} and \eqref{wh}, we have
\begin{align*}
&(ge^{-B(x+1)^{-\frac{\lambda}{2}}}\pm\phi)(x,0) =0,\\
&ge^{-B(x+1)^{-\frac{\lambda}{2}}}\pm\phi\rightarrow0\quad \text{as}\,\,h\rightarrow\infty.
\end{align*}

We claim that $ge^{-B(x+1)^{-\frac{\lambda}{2}}}\pm\phi\geq 0.$ Otherwise, by the initial and boundary conditions above, a negative minimum is obtained at some point $(x_0,\psi_0)\in(0,x_0]\times (0,+\infty)$ with $(ge^{-B(x+1)^{-\frac{\lambda}{2}}}\pm\phi)(x_0,\psi_0)<0$. In the following, we work in the domain $(0,x_0]\times (0,+\infty).$

%Note that in $[0,x_0]\times [0,+\infty)$, $ge^{-B(x+1)^{-\frac{\lambda}{2}}}\pm\phi\rightarrow 0,\,\psi\rightarrow +\infty.$ We consider in $(0,x_0]\times (0,+\infty)$.

By Lemma \ref{lem:com} and \eqref{wh}, we have
\begin{align}\label{sN}
 C\geq \sqrt{w}\geq c\frac{\psi^{\frac{1}{2}}}{(x+1)^{\frac{1}{4}}}, \quad h\leq \frac{1}{2}.
 \end{align}
% and a positive constant $a_0$ depending on $N$ such that
% \begin{align}\label{bN}
% C\geq \sqrt{w}\geq c\sqrt{\bar{w}}\geq a_0, \quad h\geq\frac{1}{N}.
% \end{align}

By \eqref{sN}, in $\{h<\frac{1}{N}\},$ we have
\begin{align}\nonumber
Lg=C_{B}N^{1-\alpha}(x+1)^{-\frac{1}{2}-\frac{1-\alpha}{2}-1}\psi^{1-\alpha}
\Big(-\frac{1}{2}-\frac{1-\alpha}{2}+(1-\alpha)\alpha \sqrt{w}\frac{x+1}{\psi^2}\Big)>0,
\end{align}
by taking $N$ large. \textbf{\it{Note $N$ is independent of choice of $B.$}} Then in $\{h<\frac{1}{N}\},$
\beno
L(ge^{-B(x+1)^{-\frac{\lambda}{2}}})=(Lg)e^{-B(x+1)^{-\frac{\lambda}{2}}}+
ge^{-B(x+1)^{-\frac{\lambda}{2}}}B\frac{\lambda}{2}(x+1)^{-\frac{\lambda}{2}-1}>0.
\eeno
Hence, the minimum point $(x_0,\psi_0)$ cannot be in $\{h<\frac{1}{N}\}.$

Consider the case of $\{h>\frac{1}{N}\}.$ We first derive the equation for $Lg+Ag.$
 Since $g=\frac{C_B}{b_0}\partial_\psi \bar{w}$ for  $h>\frac{1}{N},$ a straight calculation gives
\begin{align*}
    \partial_xg-\sqrt{\bar{w}}\partial_\psi^2g-\frac{\partial_\psi^2\bar{w}}{2\sqrt{\bar{w}}}g=0.
\end{align*}
For the third term, we have
\begin{align*}
   -\frac{\partial_\psi^2\bar{w}}{2\sqrt{\bar{w}}}&=-2\frac{\bar{u}_{yy}|_{(x,y(\psi;\bar{u}))}}{\bar{u}|_{(x,y(\psi;\bar{u}))}(\bar{u}|_{(x,y(\psi;\bar{u}))}+u|_{(x,y(\psi;u))})}
\frac{\bar{u}|_{(x,y(\psi;\bar{u}))}(\bar{u}|_{(x,y(\psi;\bar{u}))}+u|_{(x,y(\psi;u))})}{2\bar{u}^2|_{(x,y(\psi;\bar{u}))}}
\\&=A\Big(1+\frac{\phi}{2\sqrt{\bar{w}}(\sqrt{w}+\sqrt{\bar{w}})}\Big).
\end{align*}
Hence, we obtain
\begin{align}\label{LgAg}
\begin{split}&\partial_xg-\sqrt{w}\partial_\psi^2g+Ag
\\&=\frac{\sqrt{w}-\sqrt{\bar{w}}}{\sqrt{\bar{w}}}\Big(-A\frac{\phi}{
2\sqrt{\bar{w}}(\sqrt{w}+\sqrt{\bar{w}})}g-Ag-\partial_xg\Big)-A\frac{\phi}{2\sqrt{\bar{w}}(\sqrt{w}+\sqrt{\bar{w}})}g.
\end{split}
\end{align}
Now we use the properties of $f(\zeta)$ to estimate $\partial_x g.$
By \eqref{eq:EV},  we have
\begin{align*}
\partial_{\tilde{x}}g|_{(x,\psi)}=& \frac{2C_{B}}{b_0}\Big[\partial_x(\frac{1}{\sqrt{x+1}}f'')+\frac{1}{\sqrt{x+1}}\frac{ -\int_0^{y}\bar{u}_xdy'}{\bar{u}}\partial_{y}f''\Big]|_{\zeta=\frac{y(\psi;\bar{u})}{\sqrt{x+1}}}\\
=& \frac{2C_{B} }{b_0}\Big[-\frac{1}{2}\frac{1}{(\sqrt{x+1})^{3}}f''-\frac{1}{2}\frac{y}{(\sqrt{x+1})^{4}}f'''
\\&+\frac{1}{(\sqrt{x+1})^2}\frac{ -\int_0^{y}-\frac{1}{2}\frac{\bar{y}}{(\sqrt{x+1})^{3}}f''d\bar{y}}{\bar{u}}f'''\Big]|_{\zeta=\frac{y(\psi;\bar{u})}{\sqrt{x+1}}}.
\end{align*}
 For the last term,
\begin{align*}
    \int_0^{y}\frac{\bar{y}}{(\sqrt{x+1})^{3}}f''d\bar{y}
=\frac{1}{\sqrt{x+1}} \int_0^{\zeta}\zeta f''d\zeta=\frac{1}{\sqrt{x+1}}( \zeta f'(\zeta)-f(\zeta)).
\end{align*}
Hence, by $f'''=-\frac{1}{2}ff'',$
\begin{align*}
\partial_{\tilde{x}}g=-\frac{1}{2(x+1)}g+\frac{1}{4}\frac{1}{x+1}\zeta f(\zeta) g
-\frac{1}{4}\frac{1}{x+1}\frac{\zeta f'(\zeta)-f(\zeta)}{\bar{u}}f (\zeta) g|_{\zeta=\frac{y(\psi;\bar{u})}{\sqrt{x+1}}},
\end{align*}
which along with \eqref{LgAg} gives
\begin{align*}
&\partial_xg-\sqrt{w}\partial_\psi^2g+Ag
\\&=\frac{\sqrt{w}-\sqrt{\bar{w}}}{\sqrt{\bar{w}}}\Big(-A\frac{\sqrt{w}-\sqrt{\bar{w}}}{2\sqrt{\bar{w}}}g-Ag\Big)-A\frac{\sqrt{w}-\sqrt{\bar{w}}}{2\sqrt{\bar{w}}}g
\\&\quad-\frac{\sqrt{w}-\sqrt{\bar{w}}}{\sqrt{\bar{w}}}\Big(-\frac{1}{2(x+1)}g+\frac{1}{4}\frac{1}{x+1}\zeta f(\zeta) g
-\frac{1}{4}\frac{1}{x+1}\frac{\zeta f'(\zeta)-f(\zeta)}{\bar{u}}f(\zeta) g\Big)|_{(x,y)=(x,y(\psi;\bar{u}))}.
\end{align*}

Now we estimate $Lg+Ag.$ Thanks to $\sqrt{\bar{w}}\geq c_0$ in $\{h>\frac{1}{N}\}$ for some positive constant $c_0$ depending on $N$,
we infer that
\beno
|Lg+Ag|\le C|\sqrt{w}-\sqrt{\bar{w}}|\frac{1}{x+1}g(1+\zeta^2+\zeta)|_{(x,y)=(x,y(\psi;\bar{u}))}.
\eeno
Since $|\sqrt{w}-\sqrt{\bar{w}}|<C(x+1)^{-\lambda} $ and $|\sqrt{w}-\sqrt{\bar{w}}|<Ce^{-\zeta^2c}$ by Lemma \ref{lem:exp} and Lemma \ref{lem:alg},  we estimate $Lg+Ag$ in two different regions.

If $\zeta>(x+1)^{\frac{\lambda}{16}},$ we have
$$|Lg+Ag|\le Ce^{-\zeta^2c}\frac{1}{x+1}g(1+\zeta^2+\zeta)
\le C\frac{1}{x+1}g\zeta^{-16}\leq g C(x+1)^{-\lambda-1}.$$
If $\zeta\leq(x+1)^{\frac{\lambda}{16}},$ we have
$$|Lg+Ag|\le C(x+1)^{-\lambda}\frac{1}{x+1}g(1+(x+1)^{\frac{\lambda}{8}})
\le Cg(x+1)^{-\frac{7}{8}\lambda-1}.$$
Finally, we conclude that
\begin{align*}
&L(ge^{-B(x+1)^{-\frac{\lambda}{2}}})+Age^{-B(x+1)^{-\frac{\lambda}{2}}}\\
&=\big(Lg+Ag\big)e^{-B(x+1)^{-\frac{\lambda}{2}}}+
ge^{-B(x+1)^{-\frac{\lambda}{2}}}B\frac{\lambda}{2}(x+1)^{-\frac{\lambda}{2}-1}
\\ &\ge ge^{-B(x+1)^{-\frac{\lambda}{2}}}(-C(x+1)^{-\frac{7}{8}\lambda-1}+B\frac{\lambda}{2}(x+1)^{-\frac{\lambda}{2}-1})>0
\end{align*}
by taking $B$ large.
Hence, in $\{h>\frac{1}{N}\}$,
\beno
L(ge^{-B(x+1)^{-\frac{\lambda}{2}}}\pm \phi)+A(ge^{-B(x+1)^{-\frac{\lambda}{2}}}\pm \phi)>0.
\eeno
Then the minimum point $(x_0,\psi_0)$ cannot be in $\{h>\frac{1}{N}\}.$

On the other hand, the minimum cannot be achieved at the line $\{h=\frac{1}{N}\}$ by \eqref{psify} and $f'''\leq 0,$  since the graph here is a ridge with respect to $h$ for any fixed $x$. Therefore, there is no negative minimum point $(x_0,\psi_0)$ in the interior.
\end{proof}

\subsection{Proof of Theorem \ref{thm:decay0}}

Let us first prove the following lemma.

\begin{lemma}\label{cor1}
There exist positive constants $c$ and $C$ such that
\beno
\Big|u(x,y)-\bar{u}\Big(x,\int_0^{\int_0^yu(x,y')dy'}\frac{1}{\sqrt{\bar{w}}(x,\psi')}d\psi'\Big)\Big|\leq \frac{C}{\sqrt{x+1}}e^{-c\frac{y^2}{x+1} }\quad \text{in} \quad \R_+\times \R_+.
\eeno
\end{lemma}

\begin{proof} For any fixed $(x_0,y_0)\in  \R_+\times \R_+,$ let $\psi_0=\int_0^{y_0}u(x_0,y')dy'.$ Then we have
\begin{align*}
     u(x_0,y_0)=\sqrt{w}(x_0,\psi_0),\quad \bar{u}\Big(x_0,\int_0^{\int_0^{y_0}u(x_0,y')dy'}\frac{1}{\sqrt{\bar{w}}(x_0,\psi')}d\psi'\Big)=\sqrt{\bar{w}}(x_0,\psi_0).
\end{align*}
Since $cy(\psi;\bar{u})\leq y(\psi;u)\leq Cy(\psi;\bar{u}),$  it suffices to show that
\beno
|\sqrt{w}(x,\psi)-\sqrt{\bar{w}}(x,\psi)|\leq \frac{C}{\sqrt{x+1}}e^{-c\frac{(y(\psi;u))^2}{x+1} }\quad \text{in} \quad \R_+\times \R_+.
\eeno

 By \eqref{psiwyb} and Proposition \ref{51}, we have
\beno
&&|\phi(x,\psi)|\leq C \frac{1}{\sqrt{x+1}}h^{\frac{3}{4}}\quad \text{for}\,\,h<1,\\
&&|\phi(x,\psi)|\leq C \frac{1}{\sqrt{x+1}}e^{-c\zeta^2 }\quad\text{for}\,\, h\geq 1.
\eeno
Note that in $\{h\leq 1\},$ $\sqrt{w}(x,\psi)\sim\sqrt{\bar{w}}(x,\psi)\sim h^{\frac{1}{2}}$ and in $\{h\geq 1\},$ $\sqrt{w}\sim\sqrt{\bar{w}}\ge a_0$ for some positive constant $a_0$. Then we derive our result from $\sqrt{w}-\sqrt{\bar{w}}=\frac{\phi}{\sqrt{w}+\sqrt{\bar{w}}}$.\end{proof}

Now we prove Theorem \ref{thm:decay0}.

\begin{proof}
By Lemma \ref{cor1}, we have
\begin{align}\nonumber
\begin{split}
|u(x,y)-\bar{u}(x,y)|&\leq|u(x,y)-\bar{u}(x,\bar{y})|+|\bar{u}(x,\bar{y})-\bar{u}(x,y)|\\
&\leq \frac{C}{\sqrt{x+1}}e^{-c\frac{y^2}{x+1}}+|\partial_{\bar{y}}\bar{u}(x,\hat{y})||\bar{y}-y|.
\end{split}
\end{align}
where  $\bar{y}=\int_0^{\int_0^yu(x,y')dy'}\frac{1}{\sqrt{\bar{w}}(x,\psi')}d\psi'$ and $\hat{y}$ is between $\bar{y}$ and $y$. Thanks to $\int_0^yu(x,y')dy'=\int_0^{\bar{y}}\bar{u}(x,y')dy',$ we get by Lemma \ref{lem:com} that
$$c\bar{y}\leq y\leq C\bar{y}$$ for some positive constants $c\in(0,1)$ and $C.$ Due to $f'''\leq 0,$ we have
\begin{align*}
    \partial_{y}\bar{u}(x,\hat{y})|\bar{y}-y|\leq\frac{1}{\sqrt{x+1}}f''\Big(\frac{c\bar{y}}{\sqrt{x+1}}\Big)
|\bar{y}-y|.
\end{align*}
Since $0\leq f''(s)\leq Ce^{-cs^2},$ we only need to show that
\begin{align}\label{y-ybar}
|\bar{y}-y|\leq C+C\ln(x+1)+\f {Cy}{\sqrt{x+1}}.
\end{align}

Fix any $x>0.$ Take $\psi_1$ such that $\psi_1=\sqrt{x+1}f(1).$

$$
|y-\bar{y}|\leq \Big|\int_0^{\psi_1} \frac{1}{\sqrt{w}(x,\psi)}-\frac{1}{\sqrt{\bar{w}}(x,\psi)} d\psi\Big|
+\Big|\int_{\psi_1 }^\psi \frac{1}{\sqrt{w}(x,\psi)}-\frac{1}{\sqrt{\bar{w}}(x,\psi)} d\psi\Big|,
$$
where the second integral should be omitted if $\psi\leq \psi_1.$ For the first term, on the one hand,
\begin{align*}
    \int_0^{\psi_1} \frac{1}{\sqrt{w}}-\frac{1}{\sqrt{\bar{w}}} d\psi & \leq
 \int_0^{\psi_1} \frac{1}{\sqrt{w}}-\frac{1}{\sqrt{\bar{w}}+\frac{1}{\sqrt{x+1}}} d\psi\\& \leq \frac{C}{\sqrt{x+1}}\int_0^{\psi_1} \frac{1}{\sqrt{w}(\sqrt{\bar{w}}+\frac{1}{\sqrt{x+1}})}d\psi\\& \leq \frac{C}{\sqrt{x+1}}\int_0^{\psi_1} \frac{1}{\sqrt{\bar{w}}(\sqrt{\bar{w}}+\frac{1}{\sqrt{x+1}})}d\psi
\\&\leq C\int_0^{1} \frac{1}{\bar{u}+\frac{1}{\sqrt{x+1}}}d\zeta\leq C\int_0^{1} \frac{1}{b\zeta+\frac{1}{\sqrt{x+1}}}d\zeta
\\&\leq C+C\ln \sqrt{x+1},
\end{align*}
where we used $f''(\zeta)\geq b>0,\,\zeta\leq 1,$ Lemma \ref{lem:com}, $d\psi=\sqrt{x+1}f'(\zeta)d\zeta$ implying $\frac{d\psi}{\sqrt{x+1}\sqrt{\bar{w}}}=d\zeta$;
on the other hand,
\begin{align*}
    \int_0^{\psi_1} \frac{1}{\sqrt{\bar{w}}}-\frac{1}{\sqrt{w}} d\psi &\leq
 \int_0^{\psi_1} \frac{1}{\sqrt{\bar{w}}}-\frac{1}{\sqrt{w}+\frac{5C}{\sqrt{x+1}}} d\psi\\& \leq \frac{C}{\sqrt{x+1}}\int_0^{\psi_1} \frac{1}{\sqrt{\bar{w}}(\sqrt{w}+\frac{5C}{\sqrt{x+1}})}d\psi
\\&\leq C\int_0^{1} \frac{1}{u+\frac{5C}{\sqrt{x+1}}}d\zeta
\\&\leq C\int_0^{1} \frac{1}{b\zeta-\frac{C}{\sqrt{x+1}}+\frac{5C}{\sqrt{x+1}}}d\zeta
\\&\leq C+C\ln \sqrt{x+1}.
\end{align*}

For $\psi\geq \psi_1,$ since $C\geq \sqrt{w}\geq c \sqrt{\bar{w}}\geq cf'(1)>0$, we get by Lemma \ref{cor1} that
\begin{align*}
\Big|\int_{\psi_1 }^\psi \frac{1}{\sqrt{w}}-\frac{1}{\sqrt{\bar{w}}} d\psi\Big|\leq \frac{C}{\sqrt{x+1}}(\psi-\psi_1)\leq \frac{C}{\sqrt{x+1}}\int_{\psi_1 }^\psi \frac{1}{\sqrt{w}} d\psi\leq\frac{Cy}{\sqrt{x+1}}.
\end{align*}
This shows \eqref{y-ybar}.
\end{proof}

\section{Decay estimates of $\pa_x\phi$ and $\pa_\psi^2\phi$}

In the following sections, we study the Oleinik's solution $u$ with the initial data satisfying the assumptions in Theorem \ref{thm:decay1}.

\subsection{Concavity of $u$}\label{sub1}

\begin{lemma}\label{prop:uyy-upper}
Let $u$ be an Oleinik's solution with $u_0$ satisfying $\partial_{y}^2u_0\leq 0$. Then it holds that
$\partial_{y}^2u\leq 0$ in $\mathbf{R}_+\times \mathbf{R}_+.$
\end{lemma}

\begin{proof}
Let $$g=\partial_x w =\sqrt{w}\partial^2_{\psi}w.$$
By \eqref{PvM}, we have $g(x,\psi)=\sqrt{w}\partial^2_{\psi}w(x,\psi)=2\partial^2_{y} u(x,y).
$ Thus, we only need to show that
$$g\leq 0\quad\text{in}\quad[0,+\infty)\times \mathbf{R}_+.$$

Otherwise, assume that $\sup_{[0,+\infty)\times \mathbf{R}_+} g>\epsilon_0$ for some $\epsilon_0>0$.  We define
$$
x_1=\inf\Big\{x'\in[0,+\infty)|\exists \psi_{x'}\in \mathbf{R}_+ \,\,\text{so that}\, \,g(x',\psi_{x'})\geq\frac{\epsilon_0}{2}\Big\}.
$$
Due to $\partial_{y}^2u_0\leq 0,$ $g|_{x=0}\leq 0$. Hence, $x_1\in(0,+\infty).$

In the following, we only consider $g$ in $[0,x_1]\times \mathbf{R}_+.$
It is easy to see that
\begin{align}\nonumber
 \partial_x g-\frac{g^2}{2w}-\sqrt{w}\partial^2_{\psi}g=0 .
\end{align}
Using \eqref{PvM}, a straight calculation yields
 $$\sqrt{w}\partial_{\psi}g(x,\psi)=2u\frac{1}{u}\partial_{y}^3u(x,y)=2\partial_{y}^3u(x,y).$$
By $|\partial_{y}^3u|\leq C$ (See \cite{WZ19} for the bound of $|\partial_{y}^3u|$ or see Remark 5.5 in \cite{SWZ}), \eqref{wxinf0} and $g_+|_{\psi=0}= 0$ due to \eqref{owp}, we have
\beno
&&\sqrt{w}g_+\partial_{\psi}g\rightarrow 0,\quad\text{as}\quad\psi\rightarrow0\quad \text{and}\quad \text{as}\quad\psi\rightarrow+\infty,\\
&&\frac{\partial_\psi w }{2\sqrt{w}}g_+^2\rightarrow0\quad\text{as}\,\,\psi\rightarrow0\quad \text{and}\quad \text{as}\quad\psi\rightarrow+\infty.
\eeno
Then we get by integration by parts that
\begin{align}\nonumber
 \int_{\mathbf{R}_+}\sqrt{w}g_+\partial_{\psi}^2gd\psi=
  -\frac{1}{2}\int_{\mathbf{R}_+}\partial_{\psi}(\sqrt{w})\partial_{\psi}(g_+)^2
  d\psi-\int_{\mathbf{R}_+}\sqrt{w}(\partial_{\psi}g_+)^2d\psi
\end{align}
and
\begin{align}\nonumber
-\frac{1}{2}\int_{\mathbf{R}_+}\partial_{\psi}(\sqrt{w})\partial_{\psi}(g_+)^2
d\psi=\frac{1}{2}\int_{\mathbf{R}_+}\partial_{\psi}^2(\sqrt{w})(g_+)^2d\psi.
\end{align}
In $[0,x_1]\times \mathbf{R}_+,$ we have
\begin{align}\nonumber
\partial_{\psi}^2(\sqrt{w})=\frac{g}{2w}-\frac{1}{4}\frac{(\partial_\psi w )^2}{w^{\frac{3}{2 }}},\quad \frac{g_+^2}{2w}g\leq \frac{\epsilon_0}{4}\frac{(g_+)^2}{w}.
\end{align}
Then we have
\begin{align}\label{finaint}\begin{split}
 \frac{1}{2}\frac{d}{dx}\int_{\mathbf{R}_+}(g_+)^2d\psi+\frac{1}{8}\int_{\mathbf{R}_+}\frac{(\partial_\psi w )^2}{w^{\frac{3}{2 }}}(g_+)^2d\psi+\int_{\mathbf{R}_+}\sqrt{w}(\partial_\psi g_+)^2d\psi\leq C\int_{\mathbf{R}_+}\frac{(g_+)^2}{w}d\psi,\end{split}
\end{align} where $C$ depends on $x_1.$

By \eqref{PvM},  there exist some positive constants $c,\,m$ and $M$ so that
\begin{align}\nonumber
 M>\partial_\psi w (x,\psi)>m\quad x\in[0,x_1],\,\psi\in[0,c].
\end{align}
For  any fixed large $K$, by $w|_{\psi=0}=0$,  there exists a small positive constant $\psi_0<c$ such that for $x\in[0,x_1],\psi\in[0,\psi_0]$,
\begin{align*}
\frac{(\partial_\psi w )^2}{w^{\frac{3}{2 }}}\geq \frac{m^2}{(M)^{\frac{3}{2 }}\psi^{\frac{3}{2 }}}\geq\frac{K}{m\psi}\geq \frac{K}{w}.
\end{align*}
On the other hand, by \eqref{psify} and Lemma \ref{lem:com}, we have
\begin{align}\nonumber
   w\geq c\bar{w}\geq\tilde{c}_{1}>0\quad \text{on}\quad [\psi_0,+\infty).
\end{align}
Then we have
\begin{align*}
&C\int_{0}^{\psi_0}\frac{(g_+)^2}{w}d\psi\leq \frac{1}{8}\int_{\mathbf{R}_+}\frac{(\partial_\psi w )^2}{w^{\frac{3}{2 }}}(g_+)^2d\psi,\\
&C\int_{\psi_0}^{+\infty}\frac{(g_+)^2}{w}d\psi\leq C\int_{\mathbf{R}_+}(g_+)^2d\psi,
\end{align*}
which along with \eqref{finaint} give
\begin{align}
    \frac{1}{2}\frac{d}{dx}\int_{\mathbf{R}_+}(g_+)^2d\psi\leq C_{x_1}\int_{\mathbf{R}_+}(g_+)^2d\psi.
\end{align}
Since $g_+=0$ on $\{x=0\}\times\mathbf{R}_+$, by Gronwall's inequality, we have $g_+=0$ in $[0,x_1]\times\mathbf{R}_+,$ which is a contradiction to the definition of $x_1$, and thus the proof is completed.
\end{proof}

\begin{remark}
The proof is similar to Proposition 5.4 in \cite{SWZ}. However, a key difference is that we do not require the monotonicity. In particular, we do not require $\partial_yu_0\geq0$ in $\R_+$.
\end{remark}

\subsection{Decay estimate of $\phi$ in $\psi$}

To obtain a decay estimate of $\partial_x \phi,$ we first prove a better decay estimate of $\phi$ with respect to $\psi$ near $0,$ but at the expense of decay rate with respect to $x.$

 \begin{lemma}\label{phih}
There exist positive constants $C, M$ and a small positive constant $\alpha$ such that
 \beno
 |\phi(x,\psi)|\leq Cg \quad \text{in}\quad \R_+\times \R_+,
  \eeno
where
\beno
g= C(x+1)^{-\alpha}\left\{\begin{aligned}\frac{1}{b_1}\bar{w},&\quad h<\frac{1}{M},\\ 1,&\quad h\geq \frac{1}{M},
\end{aligned}\right.
\eeno
with  $b_1=f'^2(\zeta_0)$ and $\zeta_0=f^{-1}(\frac{1}{M}).$
  \end{lemma}

\begin{proof}
By Lemma \ref{lem:com}, taking $C$ large, we have $\pm\phi\leq g$ on ${x=0}$ and $g=0=\phi$ on $\psi=0.$ By Lemma \ref{lem:alg},  $|\phi|<C(x+1)^{-\alpha}$ for some small $\al>0$. Hence, there is no negative minimum of $g\pm\phi$ in $\{h\geq\frac{1}{M}\}.$

Thanks to $\partial_x\bar{w}-\sqrt{\bar{w}}\partial^2_{\psi}\bar{w}=0,$ we have
\begin{align*}
    \partial_x\bar{w}-\sqrt{w}\partial^2_{\psi}\bar{w}=-(\sqrt{w}-\sqrt{\bar{w}})
    \partial^2_{\psi}\bar{w}=-\frac{\partial_x\bar{w}}{\sqrt{\bar{w}}(\sqrt{w}+\sqrt{\bar{w}})}\phi=A\phi,
\end{align*}
which gives
\begin{align*}
 \partial_x\bar{w}-\sqrt{w}\partial^2_{\psi}\bar{w}+A\bar{w}=A(\bar{w}+\phi)=
    Aw.
\end{align*}
Then we get
\begin{align*}
    (\partial_x-\sqrt{w}\partial^2_{\psi})\big((x+1)^{-\alpha}\bar{w}C\big)+A(x+1)^{-\alpha}\bar{w}C=
   (x+1)^{-\alpha} AwC-\alpha(x+1)^{-\alpha-1}\bar{w}C.
\end{align*}
By \eqref{lem:A}, there exists a positive constant $\lambda_0$  such that $A\geq \frac{\lambda_0}{x+1}$ for $h\leq 1.$ Therefore, taking $\alpha$ small enough, we obtain
\beno
(\partial_x-\sqrt{w}\partial^2_{\psi})(g\pm\phi)+A(g\pm\phi)>0,\,\,0<h<\frac{1}{M}.
\eeno
 Hence, there is no  negative minimum of $g\pm\phi$ in $\{0<h<\frac{1}{M}\}.$

Since $\partial_\psi \bar{w}>0,\,h<\frac{2}{M},$ there is a ridge of $g$ at $h=\frac{1}{M}$ and thus no negative minimum is achieved at $h=\frac{1}{M}$.
  In summary, $g\pm\phi\geq 0.$
\end{proof}

\subsection{The equation of $\partial_x\phi=\phi^1$ }

Taking one $x$ derivative of \eqref{Main}, we get
\begin{align}\label{prexderM}
&\partial_x \phi^{1}-\sqrt{w}\partial_{\psi}^2 \phi^{1}+A\phi^{1}
-\frac{w_x}{2w}(\phi^{1}+A\phi)+\phi \partial_{x}A=0.
\end{align}
Now we simplify this equation. For this,
we use $\partial_{\tilde{x}}$ to denote the derivative in Von Mises coordinates $(\tilde{x},\psi)$ and $\partial_x$ for derivatives in Euler coordinates $(x,y)$ in case of confusion. By \eqref{Main} and \eqref{eq:EV}, we get
\begin{align*}
\partial_{\tilde{x}}A=&\partial_{\tilde{x}}\Big(-\frac{\partial_{\tilde{x}} \bar{w}}{\sqrt{\bar{w}}}\frac{1}{\sqrt{\bar{w}}+\sqrt{w}}\Big)
=\partial_{\tilde{x}}\Big(-\frac{\partial_{\tilde{x}} \bar{w}}{\sqrt{\bar{w}}}\Big)\frac{1}{\sqrt{\bar{w}}+\sqrt{w}}+\frac{\partial_{\tilde{x}} \bar{w}}{\sqrt{\bar{w}}}\frac{\frac{w_{\tilde{x}}}{2\sqrt{w}}+\frac{\bar{w}_{\tilde{x}}}{2\sqrt{\bar{w}}}}{(\sqrt{w}+\sqrt{\bar{w}})^2}
\\=&\frac{1}{\sqrt{\bar{w}}+\sqrt{w}}\Big[\partial_x\big(-2\frac{\bar{u}_{yy}}{\bar{u}}\big)-\frac{ \int_0^{y}\bar{u}_x(x,y')dy'}{\bar{u}}\partial_{y}\big(-2\frac{\bar{u}_{yy}}{\bar{u}}\big)
\Big]|_{(x,y)=(x,y(\psi;\bar{u}))}\\&+\frac{w_{\tilde{x}}}{2\sqrt{w}\sqrt{\bar{w}}}\frac{\partial_{\tilde{x}} \bar{w}}{(\sqrt{w}+\sqrt{\bar{w}})^2}
+\frac{\bar{w}_{\tilde{x}}}{2\bar{w}}\frac{\partial_{\tilde{x}} \bar{w}}{(\sqrt{w}+\sqrt{\bar{w}})^2}.
\end{align*}
We denote
$$D=\partial_x\big(\frac{\bar{u}_{yy}}{\bar{u}}\big)-\frac{ \int_0^{y}\bar{u}_x(x,y')dy'}{\bar{u}}\partial_{y}\big(\frac{\bar{u}_{yy}}{\bar{u}}\big).
$$
A direct calculation gives
\begin{align*}
\int_0^{y}\bar{u}_x(x,y')dy'=&-\frac{1}{2}\frac{1}{(\sqrt{x+1})^3}\int_0^{y}f''(\frac{\tilde{y}}{\sqrt{x+1}})\tilde{y}d\tilde{y}
=-\frac{1}{2}\frac{1}{\sqrt{x+1}}\int_0^{\zeta}\zeta f''(\zeta)d\zeta\\=&-\frac{1}{2}\frac{1}{\sqrt{x+1}}(f'(\zeta)\zeta-f(\zeta)),
\\ \frac{\bar{u}_{yy}}{\bar{u}}=&\frac{f^{(3)}}{(x+1)f'},
\end{align*}
and then
\begin{align}\label{xtildtox}
 \frac{  -\int_0^{y}\bar{u}_xdy'}{\bar{u}}=\frac{1}{2}\frac{1}{\sqrt{x+1}}\big(\zeta-\frac{f}{f'}\big).
\end{align}
Thus, we have
\begin{align*}
D=&\frac{1}{x+1}\Big(\frac{f^{(4)}}{f'}-\frac{f''f^{(3)}}{(f')^2}\Big)\Big(-\frac{y}{2(\sqrt{x+1})^{3}}\Big)
-\frac{f^{(3)}}{(x+1)^2f'}
\\&+\frac{1}{2}\frac{1}{(x+1)^2}\Big(\zeta-\frac{f}{f'}\Big)\Big(\frac{f^{(4)}}{f'}-\frac{f''f^{(3)}}{(f')^2}\Big)
\\=&\frac{1}{(x+1)^2}\Big(-\frac{1}{2}\frac{f}{f'}\Big)\Big(\frac{f^{(4)}}{f'}-\frac{f''f^{(3)}}{(f')^2}\Big)-\frac{f^{(3)}}{(x+1)^2f'}.
\end{align*}
Further, by \eqref{star}, we have
\begin{align*}
    \frac{f^{(4)}}{f'}-\frac{f''f^{(3)}}{(f')^2}=-\frac{1}{f'}(\frac{1}{2}f'f''-\frac{1}{4}f^2f'')
+\frac{1}{2}\frac{(f'')^2f}{(f')^2}.\end{align*}
This shows that
\begin{align*}
    D=&\frac{1}{(x+1)^2}\Big[\big(-\frac{1}{2}\frac{f}{f'}\big)\Big(-\frac{1}{f'}\big(\frac{1}{2}f'f''-\frac{1}{4}f^2f''\big)
+\frac{1}{2}\frac{(f'')^2f}{(f')^2}\Big)+\frac{1}{2}\frac{ff''}{f'}\Big].
\end{align*}
By \eqref{BE} and \eqref{blainf}, for any fixed large $L,$
\begin{align}\label{Dzeta}
   D=\frac{1}{(x+1)^2}O(\zeta)\,\,\text{for}\,\,\zeta\leq L,\quad D=\frac{1}{(x+1)^2}O(f''\zeta^3)\,\,\,\text{for}\,\,\zeta> L,
\end{align}
and thus,
\begin{align}\label{Dzeta}
   \frac{D}{f'}=\frac{1}{(x+1)^2}O(1)\,\,\text{for}\,\,\zeta\leq L,\quad \frac{D}{f'}=\frac{1}{(x+1)^2}O(f''\zeta^3)\,\,\text{for}\,\,\zeta>L.
\end{align}
By Proposition \ref{51}, we get
\begin{align}\label{phi12}
|\phi|\leq  C(x+1)^{-\frac{1}{2}}f''(\zeta)|_{\zeta=\frac{y(\psi;\bar{u})}{\sqrt{x+1}}}\quad\text{for}\,\, h\geq \frac{1}{M},
\end{align}
and by Lemma \ref{phih},  for a small positive constant $\alpha,$
\begin{align}\label{phi122}&|\phi|=\frac{\bar{w}}{(x+1)^\alpha}O(1)\quad \text{for}\,\,h<\frac{1}{M},\\
&\frac{\phi}{(\sqrt{w}+\sqrt{\bar{w}})^2}=\frac{1}{(x+1)^\alpha}O(1)\quad\text{for}\,\,h<\frac{1}{M}.
\end{align}
By Lemma \ref{lem:com}, we have
\begin{align}
A\phi\sim-\partial_x \bar{w}\frac{\phi}{(\sqrt{w}+\sqrt{\bar{w}})^2}.\label{eq:Aphi}
\end{align}

Summing up, we conclude that
\begin{align}\nonumber
\begin{split}
&\partial_x \phi^{1}-\sqrt{w}\partial_{\psi}^2 \phi^{1}+A\phi^{1}
-\frac{w_x}{2w}\phi^{1}=\frac{w_x}{w}O(A\phi)+O(A^2\phi) +\phi\frac{D}{f'}O(1).
\end{split}
\end{align}
By \eqref{phi12}-\eqref{eq:Aphi}, we have
\begin{align*}
|A\phi|\leq&C\frac{-\partial_x \bar{w}}{(x+1)^\alpha}\leq C\frac{1}{(x+1)^{1+\alpha}}h\quad\text{for}\,\,h<\frac{1}{M},\\
|A\phi|\leq& C(x+1)^{-\frac{3}{2}}f(f'')^2\leq C(x+1)^{-\frac{3}{2}}\zeta(f'')^2\quad\text{for}\,\,h\geq \frac{1}{M},
\end{align*}
where we also used $-\partial_x \bar{w}=\frac{ff''}{x+1},$ $h=f(\zeta)$.  By \eqref{Dzeta}, \eqref{phi12} and \eqref{phi122},
we get
  \begin{align*}
   &\phi\frac{D}{f'}=\frac{\phi}{f'^2}\frac{D}{f'}(f')^2=\frac{1}{(x+1)^{2+\alpha}}O(\zeta^2)\quad\text{for}\,\,\zeta\leq L,\\
   &\phi\frac{D}{f'}=(x+1)^{-\frac{5}{2}}O((f'')^2\zeta^3)\quad\text{for}\,\,\zeta> L.
\end{align*}
Finally, we arrive at
\begin{align}\label{eq:phi-1}
\begin{split}
&\partial_x \phi^{1}-\sqrt{w}\partial_{\psi}^2 \phi^{1}+A\phi^{1}
-\frac{w_x}{2w}\phi^{1}=\frac{w_x}{w}(D_4)+D_5,
\end{split}
\end{align}
where
\beno
&&D_4=\frac{1}{(x+1)^{1+\alpha}}O(\zeta^2)\quad\text{for}\,\,\zeta\leq L,\quad D_4= O(\zeta(f'')^2)(x+1)^{-\frac{3}{2}}\quad \text{for}\,\,\zeta> L,\\
&& D_5=\frac{1}{(x+1)^{2+\alpha}}O(\zeta^2)\quad \text{for}\,\,\zeta\leq L,\quad D_5=(x+1)^{-\frac{5}{2}}O((f'')^2\zeta^3)\quad\text{for}\,\,\zeta>L.
\eeno

\subsection{The equation of $\partial_x\bar{w}$}
Since $-\partial_x\bar{w}$ is useful in constructing barrier functions, here we derive an equation for $\partial_x\bar{w}$.
Taking one $x$ derivative to $\partial_x \bar{w}=\sqrt{\bar{w}}\partial_{\psi}^2 \bar{w}=2\partial_{y}^2\bar{u},$ we obtain
\begin{align}
   \partial_x (\partial_x\bar{w})-\sqrt{\bar{w}}\partial_{\psi}^2(\partial_x\bar{w})=
   \frac{(\partial_x\bar{w})^2}{2\bar{w}}.
\end{align}
Thanks to
\begin{align*}
    -\frac{\partial_x\bar{w}}{2\bar{w}}=A\frac{\sqrt{w}+\sqrt{\bar{w}}}{2\sqrt{\bar{w}}}=A(1+\frac{\phi}{2(\sqrt{w}+\sqrt{\bar{w}})\sqrt{\bar{w}}}),
\end{align*}
it holds that
\begin{align*}
     &\partial_x (\partial_x\bar{w})-\sqrt{w}\partial_{\psi}^2(\partial_x\bar{w})+\partial_x\bar{w}A\Big(1+\frac{\phi}{2(\sqrt{w}+\sqrt{\bar{w}})\sqrt{\bar{w}}}\Big)
     \\&=-(\sqrt{w}-\sqrt{\bar{w}})\partial_{\psi}^2(\partial_x\bar{w})\\
     &=-\frac{\phi}{\sqrt{w}+\sqrt{\bar{w}}}\frac{1}{\sqrt{\bar{w}}}\big(\partial_x (\partial_x\bar{w})-  \frac{(\partial_x\bar{w})^2}{2\bar{w}}\big)
     \\&=-\frac{\phi}{(\sqrt{w}+\sqrt{\bar{w}})\sqrt{\bar{w}}}\Big[\partial_x (\partial_x\bar{w})+\partial_x\bar{w}A\Big(1+\frac{\phi}{2(\sqrt{w}+\sqrt{\bar{w}})\sqrt{\bar{w}}}\Big)\Big].
\end{align*}
Due to $\partial_{\tilde{x}} (\frac{\partial_{\tilde{x}}\bar{w}}{2})=  \partial_{\tilde{x}}(\partial_{y}^2\bar{u})$, we get by \eqref{xtildtox} and \eqref{eq:EV} that
\begin{align*}
    \partial_{\tilde{x}}(\partial_{y}^2\bar{u})=&\partial_x\partial_{y}^2\bar{u}-\frac{ \int_0^{y}\bar{u}_xdy'}{\bar{u}}\partial_{y}\partial_{y}^2\bar{u}\\
    =&-\frac{1}{(x+1)^2}f^{(3)}+\frac{1}{x+1}\big(-\frac{y}{2(\sqrt{x+1})^3}\big)f^{(4)}+\frac{1}{2}\frac{1}{\sqrt{x+1}}\big(\zeta-\frac{f}{f'}\big)\frac{1}{(\sqrt{x+1})^3}f^{(4)}
\\=&-\frac{1}{(x+1)^2}f^{(3)}+\frac{1}{2}\frac{1}{\sqrt{x+1}}\big(-\frac{f}{f'}\big)\frac{1}{(\sqrt{x+1})^3}f^{(4)},
\end{align*}
where we used $\partial_{y}^2\bar{u}=\frac{1}{x+1}f'''.$ Therefore, by the properties of $f$ and \eqref{star}, for large $L$
\begin{align*}
 &\big|\partial_{\tilde{x}} (\frac{\partial_{\tilde{x}}\bar{w}}{2})\big|\leq\frac{C}{(x+1)^2}\zeta^2\quad\text{for}\,\,\zeta\leq L,\\
 &\big|\partial_{\tilde{x}} (\frac{\partial_{\tilde{x}}\bar{w}}{2})\big|\leq\frac{C}{(x+1)^2}f''\zeta^3\quad \text{for}\,\,\zeta> L.
\end{align*}
We denote $g_2=\partial_x\bar{w}$ and $D_1=\partial_{x} (\partial_x\bar{w})$. Then we find
\begin{align*}
\partial_x g_2-\sqrt{w}\partial_{\psi}^2g_2+g_2A=&
     -\frac{\phi}{(\sqrt{w}+\sqrt{\bar{w}})\sqrt{\bar{w}}}\Big[D_1+\partial_x\bar{w}A\Big(\frac{3}{2}+\frac{\phi}{2(\sqrt{w}+\sqrt{\bar{w}})\sqrt{\bar{w}}}\Big)\Big]
     \\=&
     -\frac{\phi}{(\sqrt{w}+\sqrt{\bar{w}})\sqrt{\bar{w}}}(D_1+D_2).
\end{align*}
By \eqref{phi12} and \eqref{phi122},  for small positive $\alpha<\frac{1}{2},$
\begin{align}\label{phih2}\begin{split}
&\Big|\frac{\phi}{(\sqrt{w}+\sqrt{\bar{w}})\sqrt{\bar{w}}}\Big|\leq \frac{C}{(x+1)^\alpha}\quad \text{for}\,\,\zeta\leq L,\\
&\Big|\frac{\phi}{(\sqrt{w}+\sqrt{\bar{w}})\sqrt{\bar{w}}}\Big|\leq \frac{Cf''}{\sqrt{x+1}}\quad \text{for}\,\,\zeta> L.
\end{split}
\end{align}
 Thanks to $-\partial_x\bar{w}A|_{(x,\psi)}\sim\frac{(ff''(\zeta))^2}{(x+1)^2}\frac{1}{(f'(\zeta))^2},$ by \eqref{phih2}, we have
\begin{align*}
     D_2&=\frac{1}{(x+1)^2}O(\zeta^2)=\frac{1}{(x+1)^2}O(h)\quad \text{for}\,\,\zeta\leq L,
     \\ D_2&=\frac{1}{(x+1)^2}O((f'')^2\zeta^2)\quad \text{for}\,\,\zeta> L.
\end{align*}
Let $D_3=D_1+D_2$, which satisfies
\begin{align}\label{729}
\begin{split}
   &D_3=\frac{1}{(x+1)^2}O(\zeta^2)=\frac{1}{(x+1)^2}O(h)\quad \text{for}\,\,\zeta\leq L,\\
    &D_3=\frac{1}{(x+1)^2}O(f''\zeta^3)\quad \text{for}\,\,\zeta>L.
    \end{split}
\end{align}

\subsection{Construction of a barrier function}
Let $g=(-g_2)e^{-K(x+1)^{-\epsilon}}.$ Note $g>0$ for $\psi>0.$ By \eqref{phih2} and \eqref{729}, we get
\begin{align*}
\begin{split}
 \partial_x & g-\sqrt{w}\partial_{\psi}^2g+Ag
\geq \epsilon K(x+1)^{-\epsilon-1}(-\partial_x\bar{w})e^{-K(x+1)^{-\epsilon}}+e^{-K(x+1)^{-\epsilon}}D_6.
\end{split}
\end{align*}
where
\begin{align*}
   D_6=&\frac{1}{(x+1)^{2+\alpha}}O(\zeta^2)=\frac{1}{(x+1)^{2+\alpha}}O(h)\quad\text{for}\,\,\zeta\leq L,\\
   D_6=&\frac{1}{(x+1)^\frac{5}{2}}O((f'')^2\zeta^3)\quad\text{for}\,\,\zeta> L.
\end{align*}
Thanks to $-\partial_x \bar{w}= \frac{1}{x+1}ff'',$ we have
\begin{align}\label{xworder}
-\partial_x \bar{w}\sim\frac{1}{x+1}\zeta^2,\,\,\zeta\leq L,\,\,\,-\partial_x \bar{w}\sim\frac{1}{x+1}\zeta f'',\,\,\zeta> L.
\end{align}
 Then we infer that
\begin{align}\label{xbarw}\begin{split}
 \partial_x & g-\sqrt{w}\partial_{\psi}^2g+Ag
\geq \frac{1}{2}\epsilon K(x+1)^{-\epsilon-1}(-\partial_x\bar{w})e^{-K(x+1)^{-\epsilon}},\end{split}
\end{align}by taking $\epsilon<\alpha$ and $K$ large such that $\epsilon K$ is large.

\subsection{Decay estimate of $\pa_x\phi$}

\begin{proposition}\label{52}
There exist positive constants $C_K$, $K$ and $\epsilon$ such that
$$
|\partial_x\phi(x,\psi)|\leq C_Kg \quad \text{in}\quad \R_+\times \R_+,
$$
where $g=(-\partial_x \bar{w})e^{-K(x+1)^{-\epsilon}}.$
\end{proposition}

\begin{proof}
By the assumption, $\partial_y u_0(0)>0$ and
$\partial_y^2u_0(y)=O(y^2)$ near $ y=0$. Then near $\psi=0,$ $|\partial_x w|_{x=0}|\leq C\psi.$
On $\{x=0\},$ we have
 \begin{align}\label{inineg}
& \partial_{x}\bar{w}=-\frac{1}{x+1}ff''<0\quad\text{for}\,\,\psi>0,\\
& \partial_{x}\bar{w}=-\frac{1}{x+1}ff'' \sim \psi\quad\psi \,\, near\,\,0.
\end{align}
Hence, near $\psi=0,$
\begin{align}\label{compatb}
|\partial_x \phi|_{x=0}|\leq C\psi\leq C(-\partial_x \bar{w}|_{x=0}).
\end{align}
By \eqref{decay2inf}, \eqref{blainf} and \eqref{yyyorder}, we have
\begin{align}\label{dfi}
|\partial_x \phi|_{x=0}|\leq -C\partial_x \bar{w}|_{x=0}\quad \text{as}\,\, \psi\rightarrow+\infty.
\end{align}
Summing \eqref{inineg}, \eqref{compatb} and \eqref{dfi}, we deduce that
\begin{align}\label{ini2d}
|\partial_x \phi|_{x=0}|\leq C(-\partial_x \bar{w}|_{x=0}).
\end{align}

Note $e^{-K}\leq e^{-K(x+1)^{-\epsilon}}\leq 1.$ By \eqref{ini2d}, we may take  $C_K$ large so that $C_Kg \pm\phi^{1}(0,\psi) \geq 0$,  and  $C_Kg \pm\phi^{1}(x,0)= 0$ and  $C_Kg \pm\phi^{1}\rightarrow0$ as $h\rightarrow\infty.$ Now we claim $C_Kg \pm\phi^{1}\geq 0.$ Otherwise,   a negative minimum is obtained at some point $(x_0,\psi_0)\in(0,x_0]\times (0,+\infty)$. Let us first point out that $-\frac{w_x}{2w}\geq0.$
By \eqref{eq:phi-1} and \eqref{xbarw}, in $\{0<h\leq\frac{1}{M}\},$  taking $\epsilon<\alpha$ small, $K\gg\frac{1}{\epsilon}$,
\begin{align}\nonumber
\begin{split}
&\partial_x (C_Kg \pm\phi^{1})-\sqrt{w}\partial_{\psi}^2 (C_Kg \pm\phi^{1})\\&\qquad+A(C_Kg \pm\phi^{1})
-\frac{w_x}{2w}(C_Kg \pm\phi^{1})\\
&>-\frac{w_x}{2w}\Big[C_Kg +\frac{1}{(x+1)^{1+\alpha}}
O(\zeta^2)\Big]\\&\qquad+C_K\frac{1}{2}\epsilon K(x+1)^{-\epsilon-1}(-\partial_x\bar{w})e^{-K} +\frac{1}{(x+1)^{2+\alpha}}
O(\zeta^2) .
\end{split}
\end{align}
By \eqref{xworder},  taking $C_K$ large so that $C_K\geq Ce^K$, we have
\begin{align*}
&\partial_x (C_Kg \pm\phi^{1})-\sqrt{w}\partial_{\psi}^2 (C_Kg \pm\phi^{1})\\&\qquad+A(C_Kg \pm\phi^{1})
-\frac{w_x}{2w}(C_Kg \pm\phi^{1})>0.
\end{align*}
Therefore, the negative minimum point $(x_0,\psi_0)$ cannot be in $\{0<h\leq\frac{1}{M}\}.$

By \eqref{eq:phi-1} and \eqref{xbarw}, in $\{h\geq\frac{1}{M}\},$  taking $\epsilon<\alpha$ small, $K\gg\frac{1}{\epsilon}$,
 \begin{align}\label{step1}\begin{split}
&\partial_x (C_Kg \pm\phi^{1})-\sqrt{w}\partial_{\psi}^2 (C_Kg \pm\phi^{1})\\&\qquad+A(C_Kg \pm\phi^{1})
-\frac{w_x}{2w}(C_Kg \pm\phi^{1})\\&>-\frac{w_x}{2w}\Big[C_Ke^{-K}(-\partial_x \bar{w}) +O(\zeta(f'')^2)(x+1)^{-\frac{3}{2}}\Big]\\&\quad+\frac{1}{2}\epsilon K(x+1)^{-\epsilon-1}C_Ke^{-K}(-\partial_x\bar{w}) +(x+1)^{-\frac{5}{2}}O((f'')^2\zeta^3).\end{split}
\end{align}
By \eqref{xworder}, taking $C_K$ large so that $C_K\geq Ce^K$, we also have
\begin{align*}
&\partial_x (C_Kg \pm\phi^{1})-\sqrt{w}\partial_{\psi}^2 (C_Kg \pm\phi^{1})\\&\qquad+A(C_Kg \pm\phi^{1})
-\frac{w_x}{2w}(C_Kg \pm\phi^{1})>0.
\end{align*}
Therefore, the negative minimum point $(x_0,\psi_0)$ cannot be in $\{h>\frac{1}{M}\}$. Therefore, there is no negative minimum point $(x_0,\psi_0)$ in the interior.
\end{proof}

%\begin{remark}\label{nd}Due to Lemma \ref{771}, for any fixed $x\in[0,+\infty),$  $-\partial_{x}\bar{w}>0$ is increasing with respect to $\psi$. This is  why  we \textbf{cannot} prove Theorem \ref{52} similarly as we proved Theorem \ref{51} by gluing two barrier functions together.\end{remark}

\subsection{Proof of Theorem \ref{thm:decay0-V}}\label{lat}

By \eqref{zeps} and Proposition \ref{51}, we get
\begin{align}\label{it1phi}\begin{split}
|\phi(x,\psi)|&\leq \frac{C }{\sqrt{x+1}}e^{-C_1\frac{(y(\psi;\bar{u}))^2}{x+1}}\leq \frac{C }{\sqrt{x+1}}e^{-c\frac{\psi^2}{x+1}}.\end{split}
\end{align}
By Lemma \ref{lem:com}, \eqref{wh} and Proposition \ref{52}, we have
\begin{align}\nonumber
\begin{split}
\big|\frac{1}{\sqrt{w}}\partial_{x}\phi(x,\psi)\big|&\leq |\frac{C}{\sqrt{\bar{w}}}\partial_{x}\phi(x,\psi)|\\&\leq \frac{C }{x+1}e^{-C_1\frac{(y(\psi;\bar{u}))^2}{x+1}}\leq \frac{C }{x+1}e^{-c\frac{\psi^2}{x+1}}.\end{split}
\end{align}
By Lemma \ref{lem:com}, \eqref{wh} and Proposition \ref{51}, we have
$$\big|\frac{\phi}{\sqrt{w}}(x,\psi)|\leq C|\frac{\phi}{\sqrt{\bar{w}}}(x,\psi)\big| \leq \frac{1}{\sqrt{x+1}}e^{-c\frac{\psi^2}{x+1}}.$$
By Lemma \ref{lem:A}, we have
\begin{align}\nonumber
\big|A\frac{\phi}{\sqrt{w}}(x,\psi)\big|\leq \frac{1}{(x+1)^{\frac{3}{2}}}e^{-c\frac{\psi^2}{x+1}}.
\end{align}
Then we infer from \eqref{Main} that
\begin{align}\label{2psiphi}\begin{split}
    |\partial_\psi^2 \phi|&\leq \frac{1}{\sqrt{w}}(|\partial_{x}\phi|+|A\phi|)\\
    &\leq \frac{C}{x+1}e^{-c\frac{\psi^2}{x+1}}+\frac{1}{(x+1)^{\frac{3}{2}}}e^{-c\frac{\psi^2}{x+1}}\leq\frac{C}{x+1}e^{-c\frac{\psi^2}{x+1}}.\end{split}
\end{align}

Now we prove that
\begin{align}\label{eq:psi-decay}
|\partial_\psi \phi |\leq \frac{C}{(x+1)^{\frac{3}{4}}}e^{-c\frac{\psi^2}{x+1} }.
\end{align}

For any fixed $(x,\psi)\in \R_+\times \R_+$, take $$\hat{\psi}=\psi+(x+1)^{\frac{1}{4}}>\psi.$$ Set $\sigma_x=(x+1)^{\frac{1}{4}}.$ Then $-\psi+\hat{\psi}=\sigma_x.$
By the mean value property, there exists a point $\psi_1\in(\psi,\hat{\psi})$ such that
$$\partial_\psi \phi(x,\psi_1) =\frac{\phi(x,\hat{\psi})-\phi(x,\psi)}{\sigma_x}.$$
Since $\hat{\psi}>\psi,$ by \eqref{it1phi}, $$|\partial_\psi \phi(x,\psi_1) | \leq\frac{2 \frac{C}{\sqrt{x+1}}e^{-c\frac{\psi^2}{x+1} }}{\sigma_x}.$$
Then since $\psi_1>\psi,$ by \eqref{2psiphi},
\begin{align*}
  |\partial_\psi \phi(x,\psi)|&\leq |\partial_\psi \phi(x,\psi) -\partial_\psi \phi(x,\psi_1) |+|\partial_\psi \phi(x,\psi_1) |\\
  &\leq \frac{C}{x+1} e^{-c\frac{\psi^2}{x+1}}\sigma_x
  +\frac{\frac{2 C}{\sqrt{x+1}}e^{-c\frac{\psi^2}{x+1} }}{\sigma_x}
  \\&\leq \frac{C}{(x+1)^{\frac{3}{4}}}e^{-c\frac{\psi^2}{x+1} }.
\end{align*}

Since $\psi$ is arbitrarily chosen, we obtain the desired result.

\section{Decay estimates of high order derivatives of $w$}

In this section, we prove Theorem \ref{thm:w-decay}.

\subsection{Comparison lemma on $\partial_\psi w$}
\begin{lemma}\label{lem:com-w}
For any  $h_1>0$, there exists $x_1>0$ depending on $h_1$ such that
 \begin{align}\nonumber
  \frac{1}{2}\partial_\psi \bar{w}\leq\partial_\psi w\leq \frac{3}{2}\partial_\psi \bar{w}\quad for\,\,h\in[0,h_1] ,\,\, \,\,x\in[x_1,+\infty).
 \end{align}
 Moreover, $\partial_\psi w \ge 0$ in $\R_+\times \R_+$.
 \end{lemma}

\begin{proof}
By \eqref{Ah}, we have
\begin{align}\nonumber
|A(x,\psi)|\leq\frac{C}{x+1}\frac{-f'''(\zeta)}{(f'(\zeta))^2}\leq\frac{C}{x+1}
\frac{ff''(\zeta)}{(f'(\zeta))^2}.
\end{align}
Hence, by \eqref{Main}, Lemma \ref{lem:com}, Proposition \ref{51}, Proposition \ref{52}, \eqref{fu} and the properties of $f$, we deduce that for $h\in[0,h_1],$
\begin{align}\label{510}\begin{split}
    |\partial_\psi^2 \phi|&\leq \frac{1}{\sqrt{w}}(|\partial_{x}\phi|+|A\phi|)\\
    &\leq \frac{C}{f'(\zeta)}\Big((-\partial_x \bar{w})+\frac{1}{x+1}
\frac{ff''(\zeta)}{(f'(\zeta))^2}\partial_\psi \bar{w}\Big)\\
&\leq \frac{C}{f'(\zeta)}\Big(\frac{f(\zeta)}{\sqrt{x+1}}\partial_\psi \bar{w}+\frac{1}{x+1}
\frac{ff''(\zeta)}{(f'(\zeta))^2}\partial_\psi \bar{w}\Big)
\\&\leq\frac{ C_{h_1}}{\sqrt{x+1}}\partial_\psi \bar{w},
\end{split}
\end{align}
where $C_{h_1}$ is a positive constant depending on $h_1.$

For any fixed $(x,\psi)\in \R_+\times \R_+$, take $$\hat{\psi}=\psi+(x+1)^{\frac{1}{4}}>\psi.$$ Set $\sigma_x=(x+1)^{\frac{1}{4}}.$ Then $-\psi+\hat{\psi}=\sigma_x.$ By the mean value property, there exists a point $\psi_1\in(\psi,\hat{\psi})$ such that
$$\partial_\psi \phi(x,\psi_1) =\frac{\phi(x,\hat{\psi})-\phi(x,\psi)}{\sigma_x}.$$
Note, by \eqref{fu},  $\sqrt{x+1}\partial_\psi \bar{w}|_{(x,\psi)}= 2f''|_{\zeta=\frac{y(\psi;\bar{u})}{\sqrt{x+1}}}$ is decreasing in $h.$ Due to $\hat{\psi}>\psi,$ we get by Proposition \ref{51} that
$$|\partial_\psi \phi(x,\psi_1) | \leq\frac{C \partial_\psi \bar{w}(x,\psi)}{\sigma_x}.$$
Due to $\psi_1>\psi,$ we get by \eqref{510} that
\begin{align*}
  |\partial_\psi \phi(x,\psi)|&\leq |\partial_\psi \phi(x,\psi) -\partial_\psi \phi(x,\psi_1) |+|\partial_\psi \phi(x,\psi_1) |\\
  &\leq \sigma_x\frac{ C_{h_1}}{\sqrt{x+1}}\partial_\psi \bar{w}(x,\psi)
  +\frac{C \partial_\psi \bar{w}(x,\psi)}{\sigma_x}.
\end{align*}
Then for any $\epsilon\in(0,1),$ we can fix $x_1$ large depending on $h_1$ and $\epsilon,$ such that
 \begin{align}\nonumber
(1-\epsilon)\partial_\psi \bar{w}\leq\partial_\psi w\leq (1+\epsilon)\partial_\psi \bar{w},\quad \text{for}\,\,h\in[0,h_1] \,\,x\in[x_1,+\infty).
\end{align}
Since $\psi$ is arbitrarily chosen, we obtain the desired result by taking $\epsilon=\frac{1}{2}$.

Since $\partial_y^2u(x,y)\leq 0$ and the positivity of $u$,  $\partial_yu$ cannot take a negative value at some point.
Then $\partial_\psi w\geq 0$ due to $2\partial_y u =\partial_\psi w$.
\end{proof}

\subsection{Uniform estimate of $\pa_x^2\phi$ at a fixed time}

Since there is no requirement on higher derivatives of $u_0$, we take some positive constant $x_1$ as a new initial time and discuss the data on $x=x_1.$ Note, by \cite{WZ19}, $w$ is smooth in $[x_1,+\infty)\times(0,+\infty).$

\begin{lemma}\label{add1}
Under the assumption in Theorem \ref{thm:decay0}, for any $x_1\in(3,+\infty),$ there exists a constant $c$ independent of $x_1$ and two constants $B_{x_1}$ and $N_1$ depending on $x_1$ such that
\begin{align}\label{anyinitdecay}
    |\pa_x^2\phi(x_1,\psi)|\leq  B_{x_1}\left\{
\begin{aligned}
\frac{1}{N_1}e^{-\frac{N_1^2}{x_1+1}c}\psi,\quad&\psi\in [0,N_1),\\
e^{-\frac{\psi^2}{x_1+1}c},\quad&\psi\in [N_1,+\infty).
\end{aligned}
\right.
\end{align}
\end{lemma}

\begin{proof}
Let $\phi^{(2)}=\pa_x^2\phi$. In $D=[x_1-2,x_1+2]\times[0,+\infty),$ by \cite{WZ19},
\begin{align}\label{HP1}
    |\phi^{(2)}|\leq|\partial_x^2 w|+|\partial_x^2 \bar{w}|\leq C_{x_1}\psi.
\end{align}
Hence, we only need to consider in $D_1=[x_1-2,x_1+2]\times[N_1,+\infty)$ for $N_1>3$ large enough determined later.

Firstly, by \eqref{it1phi}, for some positive constants $c_0$ and $C$ independent of $x_1,$
\begin{align}\label{a0bd}
|\phi|\leq C e^{-\frac{\psi^2}{x+1}c_0},\quad \psi\in[  \sqrt{x_1+2},+\infty),\,\,x\in[x_1-2,x_1+2].
\end{align}

Take $N_1$ big enough such that $$N_1>  \sqrt{x_1+2}+2.$$
For any $\psi_1\in[N_1,+\infty),$ $(x_2,\psi_2)\in[x_1-1,x_1+1]\times[\psi_1-1,\psi_1+1]$,
 we consider
\begin{align}\nonumber
\partial_x (a_0\phi)-\sqrt{w}\partial_{\psi}^2 (a_0\phi)+A(a_0\phi)=0, \quad A=-\frac{\partial_x \bar{w}}{\sqrt{\bar{w}}(\sqrt{\bar{w}}+\sqrt{w})},
\end{align} in $[x_2-1,x_2+1]\times[\psi_2-1,\psi_2+1]$, where $a_0=e^{\frac{\psi_2^2}{x_2+1}c}$ with $c=\frac{1}{8}c_0$  independent of $x_1.$ By \eqref{a0bd}, we have
\begin{align}\label{a0bd1}
|a_0\phi|\leq |e^{\frac{\psi^2}{x+1}c_0}\phi|\leq C,
\end{align}where $C$ is independent of $x_1.$

Thanks to $w(x,\psi)\rightarrow1$ and  $\bar{w}(x,\psi)\rightarrow
1$ as $\psi\rightarrow +\infty$, for $N_1$ big depending on $x_1,$ we have
\begin{align}\label{unibd}
    \frac{1}{2}\leq\bar{w},w\leq \frac{3}{2}\quad in \quad D_1.
\end{align}
Moreover, by \cite{WZ19}, we know that any  derivative of $w$ has uniform upper and lower bounds $C_{x_1}$ and $-C_{x_1}$ respectively depending on $x_1$ in $D_1.$
Then by Schauder estimates in $[x_2-1,x_2+1]\times[\psi_2-1,\psi_2+1]$ and \eqref{a0bd1}, we have
$$|a_0\phi|_{C^{1,\alpha}([x_2-\frac{1}{2},x_2+\frac{1}{2}]\times[\psi_2-\frac{1}{2},\psi_2+\frac{1}{2}] )}\leq C_{x_1} $$ where $ C_{x_1}$ depends on $x_1$ and then
\begin{align}\label{1dd}
|a_0\phi^{(1)}|_{C^{\alpha}([x_2-\frac{1}{2},x_2+\frac{1}{2}]\times[\psi_2-\frac{1}{2},\psi_2+\frac{1}{2}]
)}\leq C_{x_1}. \end{align}

 Next we consider the equation of $g=a_0\phi^{(1)}$:
\begin{align}\nonumber
\begin{split}
&\partial_x g-\sqrt{w}\partial_{\psi}^2g+Ag
-\frac{w_x}{2w}g+\partial_x A\phi a_0-\frac{w_x}{2w}A(\phi a_0)=0
\end{split}
\end{align}
in $\big[x_2-\frac{1}{2},x_2+\frac{1}{2}\big]\times\big[\psi_2-\frac{1}{2},\psi_2+\frac{1}{2}\big].$
As before, by the uniform positive upper and lower bounds of $w $ and $\bar{w}$, the bounds of derivatives of $w$ depending on $x_1$ and \eqref{1dd}, we have
\begin{align}\nonumber
|g|_{C^{1,\alpha}([x_2-\frac{1}{4},x_2+\frac{1}{4}]\times[\psi_2-\frac{1}{4},\psi_2+\frac{1}{4}],
)}\leq C_{x_1},
\end{align}
which gives
\begin{align}\nonumber
|\phi^{(2)}|\leq C_{x_1} e^{-\frac{\psi_2^2}{x_2+1}c}\quad \text{in}\quad \big[x_2-\frac{1}{4},x_2+\frac{1}{4}\big]\times\big[\psi_2-\frac{1}{4},\psi_2+\frac{1}{4}\big].
\end{align}
In particular, at $(x_2,\psi_2),$ we have
\begin{align*}
|\phi^{(2)}|\leq C_{x_1} e^{-\frac{\psi^2}{x+1}c}.
\end{align*}
Since $(x_2,\psi_2)$ is an arbitrary point in $[x_1-1,x_1+1]\times[\psi_1-1,\psi_1+1]$ and $\psi_1$ is an arbitrary value in $[N_1,+\infty),$ we conclude that
\begin{align}\label{511}
|\phi^{(2)}|\leq C_{x_1} e^{-\frac{\psi^2}{x+1}c}\quad in \quad [x_1-1,x_1+1]\times[N_1,+\infty).
\end{align}
By restricting \eqref{511} to $x=x_1,$ we have the desired result.
\end{proof}

\subsection{Decay estimate of $\pa_x^2w$}
In this subsection, we discuss under the assumptions in Theorem \ref{thm:w-decay}. Note, by \cite{WZ19}, for any $X>0,$ we can take any order derivative of $w$ in $(0,X)\times(0,+\infty).$
\begin{proposition}\label{Tm73pp}
There exist large positive constants $x_1,h_1,B$ and  small positive constant $\varepsilon$ such that
\begin{align}\label{x1decaypp}
    |\partial_x^2 w(x,\psi)|\leq  B\big(e^{-(x+1)^{-\frac{1}{2}}}\partial_\psi w+g) \quad \text{in} \quad [x_1,+\infty\big)\times[0,+\infty),
\end{align} where
$$ g(x,\psi)=\frac{1}{(x+1)^2}\left\{
\begin{aligned}
0,\quad& 0\leq h<h_1-\frac{3}{2}\pi,\\
\cos(h-h_1),\quad&h_1-\frac{3}{2}\pi\le h< h_1,\\ e^{h_1^2\varepsilon}e^{-\frac{\psi^2}{x+1}\varepsilon},\quad&h\geq h_1.
\end{aligned}
\right.
$$
In particular, for any positive constant $h_2,$ there exist positive constants $x_2$ and $B_2$ such that
\begin{align}\label{x1decaypp}
    |\partial_x^2 w(x,\psi)|\leq  \frac{B_2}{\sqrt{x+1}} \quad \text{in}\,\, h\in[0,h_2],\,\,x\in[x_2,+\infty).
\end{align}
\end{proposition}

\begin{proof}
We first determine $h_1 $ and $\varepsilon$ which are independent of $x_1.$ Let
$$L_0v=\partial_x v-\sqrt{w}\partial_\psi^2 v,$$
and $h_1\in(100,+\infty), \varepsilon\in(0,c)$ where $c$ is the constant independent of $x_1$ in Lemma \ref{add1}.

By Lemma \ref{lem:com}, $ w\leq C$ with $C$ independent of $h_1$ and $x_1$. Taking  $\varepsilon$ to be a small positive constant  independent of $h_1$ and $x_1$, we get
\begin{align*}
L_0\Big(\frac{1}{(x+1)^2}e^{-\frac{\psi^2}{x+1}\varepsilon}\Big)&=\frac{1}{(x+1)^2}e^{-\frac{\psi^2}{x+1}\varepsilon}\Big[-\frac{2}{x+1}+\varepsilon\frac{\psi^2}{(x+1)^2}
 -\sqrt{w}\Big(-\varepsilon\frac{2}{x+1}+\varepsilon^2\frac{4\psi^2}{(x+1)^2}\Big)\Big]
\\&\geq \frac{1}{(x+1)^2}e^{-\frac{\psi^2}{x+1}\varepsilon}\Big(-\frac{2}{x+1}+\frac{\varepsilon}{2}\frac{\psi^2}{(x+1)^2}
 \Big).
 \end{align*}
This shows that for $h\geq h_1,$
\begin{align*}
L_0\Big(\frac{1}{(x+1)^2}e^{-\frac{\psi^2}{x+1}\varepsilon}\Big)&\geq \frac{1}{(x+1)^2}e^{-\frac{\psi^2}{x+1}\varepsilon}\Big(-\frac{2}{x+1}+\frac{\varepsilon}{2}\frac{h_1^2}{x+1}
 \Big).
 \end{align*}
 Fix $h_1$ large independent of $x_1$ such that $\frac{\varepsilon h_1^2}{2}>3$ and thus,
  \begin{align}\label{h1h}
 L_0g>0\quad \text{for}\,\,h>h_1.
 \end{align}

Next we will determine $ x_1$ large depending on $h_1$ and $B$ large depending on $ x_1.$
For any $X>x_1,$ we work in $[x_1,X]\times [0,+\infty).$ By Lemma \ref{add1}, we have
 \begin{align}\label{xxbipp}
 \partial_x^2 w=0\quad \text{on} \quad\psi=0,\quad \partial_x^2 w\rightarrow0\quad \text{as} \quad\psi\rightarrow+\infty,
\end{align}
and for $x_1$ and $B$ large enough,
\begin{align*}
    |\partial_x^2 w|\leq B\big(e^{-(x+1)^{-\frac{1}{2}}}\partial_\psi w+g\big)\quad \text{on} \quad x=x_1.
\end{align*}
Indeed, by Lemma \ref{lem:com} and \eqref{fu}, for some positive constants $c_{h_1}$ and $x_1$ depending on $h_1, $ it holds that for $h\in[0,h_1],\,\, x\in[x_1,+\infty),$
 \begin{align}\label{714}\begin{split}
 \partial_\psi w(x,\psi)\geq&\frac{1}{2}\partial_\psi \bar{w}(x,\psi)\\
 \geq&\frac{1}{\sqrt{x+1}}f''|_{\zeta=\frac{y(\psi;\bar{u})}{\sqrt{x+1}}}\geq
 \frac{c_{h_1}}{\sqrt{x+1}}.\end{split}
\end{align}
 Since  $|g|\leq\frac{1}{(x+1)^2}$ and
 \begin{align}\label{e-1}
 e^{-1}\leq e^{-(x+1)^{-\frac{1}{2}}}\leq 1,
\end{align}
we deduce that for $x_1$ large depending on $h_1,$ for some small positive constant $c$ depending on $ h_1,$ it holds
$$e^{-(x+1)^{-\frac{1}{2}}}\partial_\psi w+g\geq \frac{c}{\sqrt{x+1}},\quad h\in[0,h_1],\,\, x\in[x_1,+\infty).$$
 Then taking $B$ large depending on $x_1$ and $h_1 $, we get by Lemma \ref{add1} that
 \begin{align}\label{Bx1}
 |\partial_x^2 w|\leq B(e^{-(x+1)^{-\frac{1}{2}}}\partial_\psi w +g)\quad \text{on} \quad x=x_1.
 \end{align}

Now we work in the interior domain $(x_1,X]\times (0,+\infty)$. We derive the equations of $F=\partial_x^2 w$ and $\varphi +g$, where $\varphi= e^{-(x+1)^{-\frac{1}{2}}}\partial_\psi w.$ Direct calculations show that
\begin{align*}
&\partial_x F-\sqrt{w}\partial_{\psi}^2F
-\frac{3\partial_xw}{2w}F=-\frac{3(\partial_xw)^3}{4w^2}.
\end{align*}
Let
$$L_1v=\partial_x v-\sqrt{w}\partial_\psi^2 v-\frac{3\partial_xw}{2w}v.$$
Note $-\frac{3\partial_xw}{2w}\geq0.$ We infer that
\begin{align}\label{signf}
L_1(-F)\geq -\frac{3\partial_xw}{2w}\Big(-\frac{(\partial_xw)^2}{2w}\Big).
\end{align}
From Proposition \ref{52}, Lemma \ref{lem:com} and \eqref{wh}, we deduce that
for some positive constants $c$ and $C$ independent of $x_1,$
\begin{align}\label{flower}
\frac{ (\partial_xw)^2}{2w}\leq\frac{C}{(x+1)^2}\left\{
\begin{aligned}
h,\quad&h\leq 1,\\
1,\quad&1< h\le 2,\\
e^{-c\frac{\psi^2}{x+1}},\quad& 2<h.
\end{aligned}
\right.
\end{align}

By a direct calculation, we have
\begin{align*}
  \partial_x(   \partial_\psi w)-\sqrt{w}\partial_\psi^2(   \partial_\psi w)- \frac{\partial_xw}{2w}\partial_\psi w=0.
\end{align*}
Hence, for $\varphi= e^{-(x+1)^{-\frac{1}{2}}}\partial_\psi w$, by \eqref{e-1}, we get
\begin{align}\label{l1ph}
L_1\varphi\geq - \frac{\partial_xw}{w}e^{-1}\partial_\psi w
+\frac{\frac{1}{2} e^{-1}}{(x+1)^{\frac{1}{2}+1}}\partial_\psi w
\geq 0.
\end{align}
Then by \eqref{h1h}, \eqref{flower} and  \eqref{l1ph}, for $B>C$ and $h\in(h_1,+\infty),$
\begin{align*}
    L_1\big(B(\varphi+g)-F\big)> -\frac{3\partial_xw}{2w}\Big(B\frac{1}{(x+1)^2}e^{-\frac{\psi^2}{x+1}\varepsilon}-\frac{C}{(x+1)^2}e^{-c\frac{\psi^2}{x+1}}\Big)\geq0 ,
\end{align*}
and thus,
$$L_1\big(B(\varphi+g)-F\big)>0,\quad h>h_1,$$ where we used $\varepsilon\in(0,c).$

For $h\in(0,h_1-\frac{3}{2}\pi),$ by \eqref{714} and  \eqref{l1ph},
\begin{align*}
L_1\big(B(\varphi+g)\big)> - \frac{3\partial_xw}{2w}\frac{2B}{3e}\frac{c_{h_1}}{\sqrt{x+1}},
\end{align*} where we have used $
\partial_\psi w>0$ for $h\in(0,h_1).$
Hence, by \eqref{flower}, for $B$ large depending on $h_1$, it holds in $(0,h_1-\frac{3}{2}\pi)$ that
\begin{align*}
L_1(B(\varphi+g)-F)>0.
\end{align*}
 Next, for $h\in(h_1-\frac{3}{2}\pi,h_1),$
 \begin{align*}
   L_1 g= & L_1\Big(\frac{1}{(x+1)^2}\cos\big(\frac{\psi}{\sqrt{x+1}}-h_1\big)\Big)
   \geq  -\frac{C_{h_1}}{(x+1)^3}-\frac{3\partial_xw}{2w}(-\frac{1}{(x+1)^2})
\\ \geq &-\frac{C_{h_1}}{(x+1)^3} & \end{align*} where $C_{h_1}$ is a positive constant depending on $h_1.$

Hence, for $h\in(h_1-\frac{3}{2}\pi,h_1),$ by \eqref{714} and \eqref{l1ph} , we have
\begin{align*}
    L_1(\varphi+g)\geq - \frac{\partial_xw}{w}e^{-1}\partial_\psi w + \frac{c_{h_1}}{(x+1)^{\frac{1}{2}+\frac{3}{2}}} -\frac{C_{h_1}}{(x+1)^3}.
\end{align*}  Hence, for $x_1$ large depending on $h_1,$ by \eqref{714}, we have
\begin{align*}
    L_1\big(B(\varphi+g)\big)>- \frac{3\partial_xw}{2w}\frac{2B}{3e} \frac{c_{h_1}}{\sqrt{x+1}} ,\quad h\in(h_1-\frac{3}{2}\pi,h_1),\,\,x\in[x_1,+\infty).
\end{align*}
Therefore, by \eqref{flower}, for $B$ large depending on $h_1,$ we have
\begin{align*}
    L_1\big(B(\varphi+g)-F\big)>0 ,\quad h\in(h_1-\frac{3}{2}\pi,h_1),\,\,x\in[x_1,+\infty).
\end{align*}
In summary,  due to the sign of $L_1(Bg-F)$ in $0<h<h_1-\frac{3}{2}\pi$, $h_1-\frac{3}{2}\pi<h<h_1$ and $h>h_1$, $-F +B(\varphi+g)$ cannot achieve a negative minimum in $[x_1,X]\times [0,+\infty)\setminus \{h=h_1-\frac{3}{2}\pi,h_1\}.$  Moreover, since $h=h_1-\frac{3}{2}\pi$ and $h=h_1$ are two ridges of $g$, a minimum of $-F +B(\varphi+g)$ cannot be achieved at $h=h_1-\frac{3}{2}\pi$ and $h=h_1.$ Then we conclude $-F +B(\varphi+g)\geq 0.$

Similarly, we can prove $F +B(\varphi+g)\geq 0$ by following the computations above and using \eqref{flower}.
\end{proof}

\begin{proposition}\label{Tm73}
For any $\alpha\in (0,\frac{1}{8}),$ there exist large positive constants $B,\,h_1,x_1$and small positive constants $h_0,\varepsilon$ such that
\begin{align}\label{x1decay}
    |\partial_x^2 w(x,\psi)|\leq  B g \quad \text{in} \quad [x_1,+\infty)\times[0,+\infty),
\end{align}  where  $$ g(x,\psi)=\frac{1}{(x+1)^\frac{1}{2}}\left\{
\begin{aligned}
h^{1-\alpha},\quad&h\leq h_0,\\
h_0^{1-\alpha},\quad&h_0\le h\le h_1,\\
h_0^{1-\alpha} e^{h_1^2\varepsilon}e^{-\frac{\psi^2}{x+1}\varepsilon},\quad&h\geq h_1.
\end{aligned}
\right.
$$
\end{proposition}
\begin{proof}
For any $X$ such that $X>x_1,$ we work in $[x_1,X]\times [0,+\infty).$
We first determine $h_0\in(0,1)$ and $h_1.$
 By \eqref{wh}, we have
 $$ c\frac{\psi^{\frac{1}{2}}}{(x+1)^{\frac{1}{4}}}\leq \sqrt{w}\leq C\frac{\psi^{\frac{1}{2}}}{(x+1)^{\frac{1}{4}}}, \quad h\leq 1,$$
 for some constants $c$ and $C$ independent of $x_1.$ Hence,
 \begin{align*}
  L_0\Big(\frac{\psi^{1-\alpha}}{(x+1)^{\frac{1-\alpha}{2}+\frac{1}{2}}}\Big)
  &=\frac{\psi^{1-\alpha}}{(x+1)^{\frac{1-\alpha}{2}+\frac{1}{2}}}\Big[-(\frac{1-\alpha}{2}+\frac{1}{2})
  \frac{1}{x+1}+\sqrt{w}\alpha(1-\alpha)\psi^{-2}\Big]\\
  &\geq\frac{\psi^{-1-\alpha}}{(x+1)^{\frac{1-\alpha}{2}+\frac{1}{2}}}
  [-
  h^2+ch^{\frac{1}{2}}\alpha(1-\alpha)].
 \end{align*}
 Hence, we fix $h_0$ small depending on $\alpha$ such that
\begin{align*}
 L_0g>0,\quad 0<h< h_0.
 \end{align*}
Therefore, for $B>C$ where $C$ is the constant in \eqref{flower}, we get
 \begin{align}\label{Lgh0}
 L_1(Bg)>-\frac{3\partial_xw}{2w}\Big(\frac{(\partial_xw)^2}{2w}\Big),\quad 0<h< h_0.
 \end{align}

Next we determine $\varepsilon\in(0,c)$ where $c$ is the constant in Lemma \ref{add1} and then we determine $h_1$.
Since $w\leq C$ for some positive constant $C$ independent of $h_1$ and $x_1,$ taking  $\varepsilon$ to be a small positive constant  independent of $h_1$ and $x_1$, we have
\begin{align*}
L_0\Big(\frac{1}{\sqrt{x+1}}e^{-\frac{\psi^2}{x+1}\varepsilon}\Big)&=\frac{1}{\sqrt{x+1}}e^{-\frac{\psi^2}{x+1}\varepsilon}\Big[-\frac{1}{2}\frac{1}{x+1}+\varepsilon\frac{\psi^2}{(x+1)^2}
 -\sqrt{w}\Big(-\varepsilon\frac{2}{x+1}+\varepsilon^2\frac{4\psi^2}{(x+1)^2}\Big)\Big]
\\&\geq \frac{1}{\sqrt{x+1}}e^{-\frac{\psi^2}{x+1}\varepsilon}\Big(-\frac{1}{2}\frac{1}{x+1}+\frac{\varepsilon}{2}\frac{\psi^2}{(x+1)^2}\Big).
\end{align*}
This implies that for $h\geq h_1,$
\begin{align*}
L_0(\frac{1}{\sqrt{x+1}}e^{-\frac{\psi^2}{x+1}\varepsilon})&\geq \frac{1}{\sqrt{x+1}}e^{-\frac{\psi^2}{x+1}\varepsilon}(-\frac{1}{2}\frac{1}{x+1}+\frac{\varepsilon}{2}\frac{h_1^2}{x+1}
 ) .\end{align*}
 Fix $h_1$ large independent of $x_1$ such that $\frac{\varepsilon h_1^2}{2}>1,$ and thus, $L_0g>0$ for $h_1< h$.
 Then by \eqref{flower}, for $B$ large depending on $\alpha$ and $h_0$, we have
 \begin{align}\label{Lgh1}
    L_1(Bg) >-\frac{3\partial_xw}{2w}\big(\frac{(\partial_xw)^2}{2w}\big),\quad h> h_1.
 \end{align}

 By \eqref{x1decaypp}, taking $x_1$ large depending on
 $h_1$ and $B$ large depending on
 $h_1$ and $h_0^{1-\alpha},$ we have
 $$|\pa_x^2w|\leq Bg,\quad h\in[h_0,h_1],\,\,x\in[x_1,+\infty) .$$
Finally, we consider the initial and the boundary data. By Lemma \ref{add1}, we have
 \begin{align}\label{xxbi}
 \partial_x^2 w=0\quad on \quad\psi=0,\quad \partial_x^2 w\rightarrow0\quad as \quad\psi\rightarrow+\infty,\quad |\partial_x^2 w|\leq Cg \quad on \quad x=x_1,
\end{align} where $ C$ is a positive constant depending on $\alpha, h_1, x_1, h_0$ and $\varepsilon$.
Take $B\geq C.$ Then $|\partial_x^2 w|\leq Bg$ on $x=x_1$.

Summing up,  we have $L_1(Bg+F)>0$ in $[x_1,X]\times [0,+\infty)\setminus \{h=h_0, h=h_1\}$, and thus $\partial_x^2 w +Bg$ cannot achieve a negative minimum in $[x_1,X]\times [0,+\infty)\setminus \{h=h_0, h=h_1\}.$ Moreover, since $h=h_0, h=h_1$ are two ridges,  a minimum  of
$Bg+F$ cannot be achieved on $h=h_0$ and $h=h_1$. Then we conclude $Bg+F\geq 0.$

Similarly, we can prove $Bg-F\geq 0.$
 \end{proof}

 \begin{remark}(1) Due to the change of the structure of the equation, a key difficulty in this subsection is how to derive good terms without using the good term $A$ as before. (2)  It is necessary to distinguish which terms depend on $x_1.$\end{remark}

 \subsection{Proof of Theorem \ref{thm:w-decay}}

 First of all, we infer from Proposition \ref{Tm73} that
 \begin{align*}
 |\partial_x^2 w|\leq \frac{C}{(x+1)^{\frac{1}{2}}}e^{-c\frac{\psi^2}{x+1} }.
\end{align*}
Note that
  \begin{align*}
   \partial_x\partial_\psi^2 w=\partial_x(\frac{\partial_x w}{\sqrt{w}})=\frac{\partial_x^2 w}{\sqrt{w}}-\frac{1}{2}\frac{(\partial_x w)^2}{w^{\frac{3}{2}}}.
\end{align*}
Then by  Lemma \ref{lem:com}, \eqref{wh}, Proposition \ref{52} and Proposition \ref{Tm73}, we get
\begin{align}\nonumber
 |\partial_x\partial_\psi^2 w|\leq \frac{C}{(x+1)^{\frac{1}{2}}}e^{-c\frac{\psi^2}{x+1}},\quad (x,\psi)\in (N,+\infty)\times (0,+\infty).
\end{align}

Following the argument in section 4.7 but taking $\sigma_x=(x+1)^{-\frac{1}{4}}$, we can show that
\begin{align}\nonumber
  |\partial_{\psi x} w| \leq \frac{C}{(x+1)^{\frac{3}{4}}}e^{-c\frac{\psi^2}{x+1} },\quad (x,\psi)\in (N,+\infty)\times (0,+\infty).
\end{align}

\section{Decay estimates of high order derivatives of $u$}\label{yue}

In this section, we prove Theorem \ref{thm:decay1}, which is a direct consequence of the following Proposition \ref{prop:u-py}, Proposition \ref{prop:u-px} and Proposition \ref{prop:u-pxy}.

\subsection{Decay estimates of $\pa_y u$ and $\pa_y^2u$}

\begin{proposition}\label{prop:u-py}
There exist positive constants $c$ and $ C$ such that for any $(x,y)\in \R_+\times \R_+$,
\begin{align}\label{28}
  -\frac{C}{x+1}e^{-c\frac{y^2}{x+1}}\leq \partial_y^2u(x,y)\leq 0,\end{align} and \begin{align*}
|\partial_y (u(x,y)-\bar{u}(x,y))|\leq \frac{C}{(x+1)^{\frac{3}{4}}}\ln(x+e)e^{-c\frac{y^2}{x+1} }.
\end{align*}
\end{proposition}

\begin{proof}
By \eqref{PvM}, we have
\begin{align}\nonumber
2\partial^2_{y} u(x,y)=\sqrt{w}\partial^2_{\psi}w(x,\psi(x,y))=\partial_{x}w(x,\psi(x,y)),
\end{align}
where $\psi(x,y)=\int_0^yu(x,y')dy'.$
It follows from Proposition \ref{52} that
\begin{align}\nonumber
\begin{split}
   0&\geq\partial_{x}w|_{(x,\psi)}\geq C\partial_{x}\bar{w}|_{(x,\psi)}\\ &\geq
   -\frac{C}{x+1}ff''|_{\zeta=\frac{y(\psi;\bar{u})}{\sqrt{x+1}}}\geq
   -\frac{C}{x+1}e^{-C_1\zeta^2}|_{\zeta=\frac{y(\psi;\bar{u})}{\sqrt{x+1}}}
   \\&\geq -\frac{C}{x+1}e^{-c\frac{(y(\psi;u))^2}{x+1}},
   \end{split}
\end{align}
where we used \eqref{yyyorder}. Hence, we have \eqref{28}.

For any fixed $(x,y)\in \R_+\times \R_+$, take $$\hat{y}=y+(x+1)^{\frac{1}{4}}>y.$$ Set $\sigma_x=(x+1)^{\frac{1}{4}}$ and $\varphi(x,y)=u(x,y)-\bar{u}(x,y).$ Then $-y+\hat{y}=\sigma_x.$
By the mean value property, there exists a point $y_1\in(y,\hat{y})$ such that
$$\partial_y \varphi(x,y_1) =\frac{\varphi(x,\hat{y})-\varphi(x,y)}{\sigma_x}.$$
Since $\hat{y}>y,$ we get by Theorem \ref{thm:decay0} that
$$|\partial_y \varphi(x,y_1) | \leq\frac{2 \frac{C}{\sqrt{x+1}}e^{-c\frac{y^2}{x+1} }}{\sigma_x}\ln(e+x).$$
Then by $y_1>y$ and  \eqref{28}, we deduce that
\begin{align*}
  |\partial_y \varphi(x,y)|&\leq |\partial_y \varphi(x,y) -\partial_y \varphi(x,y_1) |+|\partial_y \varphi(x,y_1) |\\
  &\leq C\frac{1}{x+1} e^{-c\frac{y^2}{x+1}}\sigma_x
  +\frac{2 \frac{C}{\sqrt{x+1}}e^{-c\frac{y^2}{x+1} }}{\sigma_x}\ln (x+e)
  \\&\leq C\frac{1}{(x+1)^{\frac{3}{4}}}\ln(x+e)e^{-c\frac{y^2}{x+1}}.
\end{align*}
Since $y$ is arbitrarily chosen, we obtain the desired result.
\end{proof}

\subsection{Decay estimate of $\pa_x u$}

\begin{proposition}\label{prop:u-px}
There exist positive constants $c$ and $  C$ such that for any $(x,y)\in \R_+\times \R_+$,
$$|\partial_x (u(x,y)-\bar{u}(x,y))|\leq \frac{C}{x+1}e^{-c\frac{y^2}{x+1}}.$$
\end{proposition}

\begin{proof}
We use $\partial_x$ to denote the derivative in Euler coordinates $(x,y)$ and $\partial_{\tilde{x}}$ to denote the derivative in  Von Mises coordinates $(\tilde{x},\psi).$ We denote
\begin{align*}
    h(x,y)=w(x,\psi_2(x,y))-\bar{w}(x,\psi_2(x,y)),\quad g(x,y)=\bar{w}(x,\psi_2(x,y)),
\end{align*}
where $\psi_2(x,y)=\int_0^y u(x,y')dy'.$ Here we note, by the composition, the independent variables of both $w(x,\psi_2(x,y))$ and $\bar{w}(x,\psi_2(x,y))$ are $x$ and $y.$
By the definition of $\psi_2(x,y),$ we have $$w(x,\psi_2(x,y))=u(x,y)^2,$$
and thus,
$$\partial_x(u^2-\bar{u}^2)=\partial_x h+\partial_x(g-\bar{u}^2).
$$
By a straight computation, we have
\begin{align*}
    \partial_x h|_{(x,y)}
=&(\partial_{\tilde{x}}w-\partial_{\tilde{x}}\bar{w})|_{(x,\psi)=(x,\psi_2(x,y))}
+\partial_x\psi_2|_{(x,y)}(\partial_\psi w-\partial_\psi\bar{w})|_{(x,\psi)=(x,\psi_2(x,y))},\\
\partial_x g|_{(x,y)}
=&\partial_{\tilde{x}}\bar{w}|_{(x,\psi)=(x,\psi_2(x,y))}
+\partial_x\psi_2|_{(x,y)}\partial_\psi\bar{w}|_{(x,\psi)=(x,\psi_2(x,y))},
\end{align*}
and
\begin{align*}
\partial_x(u^2-\bar{u}^2)=2u\partial_x u-2\bar{u}\partial_x\bar{u}=2u\partial_x (u-\bar{u})+2(u-\bar{u})\partial_x\bar{u}.
\end{align*}
Then we obtain
 \begin{align}\label{xu1}
 \begin{split}
    \partial_x (u-\bar{u})|_{(x,y)}=&\frac{1}{2u|_{(x,y)}}\Big[\partial_{\tilde{x}} \phi|_{(x,\psi)=(x,\psi_2(x,y))}+\partial_x\psi_2|_{(x,y)}\partial_{\psi} \phi|_{(x,\psi)=(x,\psi_2(x,y))}\\&+\partial_{\tilde{x}}\bar{w}|_{(x,\psi)=(x,\psi_2(x,y))}
+\partial_x\psi_2|_{(x,y)}\partial_\psi\bar{w}|_{(x,\psi)=(x,\psi_2(x,y))}\Big]
\\&-\frac{\bar{u}}{u}\partial_x\bar{u}|_{(x,y)}-\frac{1}{u}(u-\bar{u})\partial_x\bar{u}|_{(x,y)}
.\end{split}
\end{align}

Now we estimate each term on the right hand side of \eqref{xu1}.
First of all, we show that for some positive constants $c_9$ and $C_9$,
\begin{align}\label{ubaruxy}
c_9\bar{u}(x,y)\leq u(x,y)\leq C_9\bar{u}(x,y).
\end{align}

Since $cy(\psi;\bar{u})\leq y(\psi;u)\leq Cy(\psi;\bar{u})$ and $y(\psi_2(x,y);u)=y,$ we get by Lemma \ref{lem:com} that
\begin{align*}
    u(x,y)=&\sqrt{w}(x,\psi_2(x,y))\\ \leq & C\sqrt{\bar{w}}(x,\psi_2(x,y))= C\bar{u}(x,y(\psi_2(x,y);\bar{u}))
   \\ \leq& C\bar{u}(x,Cy(\psi_2(x,y);u))=C\bar{u}(x,Cy)\leq C_9\bar{u}(x,y),
\end{align*}
and
\begin{align*}
    u(x,y)=&\sqrt{w}(x,\psi_2(x,y))\\ \geq & c\sqrt{\bar{w}}(x,\psi_2(x,y))= c\bar{u}(x,y(\psi_2(x,y);\bar{u}))
   \\ \geq& c\bar{u}(x,cy(\psi_2(x,y);u))=c\bar{u}(x,cy)\geq c_9\bar{u}(x,y).
\end{align*}
This proves \eqref{ubaruxy}.

From \eqref{ubaruxy} and $\bar{u}(x,y)=f'(\frac{y}{\sqrt{x+1}})$, we infer that
\beno
 \frac{1}{u(x,y)}|\partial_x\bar{u}(x,y)|\leq \frac{C}{x+1}e^{-c\frac{y^2}{x+1}}.
 \eeno
 By Proposition \ref{52}, we have
\begin{align}\label{xuphi}\begin{split}
    \frac{1}{u(x,y)}|\partial_{\tilde x} \phi|_{(x,\psi)=(x,\psi_2(x,y))} |\leq & \frac{C}{f'}\frac{ff''}{x+1}|_{(x,y)=(x,y(\psi_2(x,y);\bar{u}))}
    \\ \leq& \frac{C}{x+1}e^{-C_1\frac{\big(y(\psi_2(x,y);\bar{u})\big)^2}{x+1}}\\
    \leq & \frac{C}{x+1}e^{-c\frac{\big(y(\psi_2(x,y);u)\big)^2}{x+1}}
    = \frac{C}{x+1}e^{-c\frac{y^2}{x+1}}.\end{split}
\end{align}

 By  Lemma \ref{lem:com}, $d\psi=\sqrt{x+1}f'(\zeta)d\zeta$ and Proposition \ref{52},
 we get
\begin{align*}
  |\partial_x\psi_2(x,y)|&=\Big|\frac{1}{2}\sqrt{w}\int_0^{\psi_2(x,y)}w^{-\frac{3}{2}}
   \partial_{\tilde x}w d\psi\Big|\\&\leq C\Big|\sqrt{\bar{w}}\int_0^{\psi_2(x,y)}\bar{w}^{-\frac{3}{2}}
   \partial_{\tilde x}\bar{w} d\psi\Big|\\
   &\leq C\sqrt{x+1}f' \int_0^{\zeta}(f')^{-2}\frac{1}{x+1}ff''
    d\tilde{\zeta}|_{\zeta=\frac{y(\psi_2(x,y);\bar{u})}{\sqrt{x+1}}}
 \\& \leq  \frac{C}{\sqrt{x+1}} e^{-c\frac{(y(\psi_2(x,y);\bar{u}))^2}{x+1} }
    \\&\leq  \frac{C}{\sqrt{x+1}} e^{-c\frac{(y(\psi_2(x,y);u))^2}{x+1} }= \frac{C}{\sqrt{x+1}} e^{-c\frac{y^2}{x+1} }. \end{align*}
By \eqref{eq:psi-decay} and \eqref{zeps},  we get
$$
|\partial_\psi \phi |_{(x,\psi)=(x,\psi_2(x,y))}|\leq \frac{C}{(x+1)^{\frac{3}{4}}}
e^{-c\frac{(y(\psi_2(x,y);\bar u))^2}{x+1} }\le \frac{C}{(x+1)^{\frac{3}{4}}}e^{-c\frac{y^2}{x+1} }.
$$

Summing up the estimates above and using \eqref{fu}, we can conclude our result.
\end{proof}

\subsection{Decay estimate of $\pa_x\pa_y u$}

\begin{proposition}\label{prop:u-pxy}
There exist positive constants $c$, $ C$ and $N$ such that for any $(x,y)\in (N,+\infty)\times \R_+$,
\begin{align*}
    |\partial_{xy} u(x,y)|\leq \frac{C}{(x+1)^{\frac{3}{4}}}e^{-c\frac{y^2}{x+1}}.
\end{align*}
\end{proposition}

\begin{proof} We use $\partial_x$ to denote the derivative in Euler coordinates $(x,y)$ and $\partial_{\tilde{x}}$ to denote the derivative in  Von Mises coordinates $(\tilde{x},\psi).$
By \eqref{PvM} and \eqref{starmain}, we get
\begin{align*}
   \partial_x(2\partial_y u)&=\partial_{\psi \tilde{x}} w+\partial_x\psi\partial_{\psi}^2 w=\partial_{\psi \tilde{x}} w+\partial_x\psi\frac{\partial_{x} w}{\sqrt{w}},
  \\  \partial_x\psi&=\frac{1}{2}\sqrt{w}\int_0^\psi w^{-\frac{3}{2}}
    \partial_{\tilde{x}} wd\psi'.
\end{align*}
From Lemma \ref{lem:com}, Proposition \ref{52} and \eqref{wh}, we infer that
\begin{align*}
    |\partial_x\psi(x,y)|&\leq C\Big|\int_0^{\sqrt{x+1}}w^{-\frac{3}{2}}
    \partial_{\tilde{x}} wd\psi+\int_{\sqrt{x+1}} ^{+\infty} w^{-\frac{3}{2}}
    \partial_{\tilde{x}} wd\psi\Big|
    \\&\leq C\frac{1}{\sqrt{x+1}}\int_0^{1}h^{-\frac{3}{2}}hdh+\frac{C}{\sqrt{x+1}}
    \int_1^{+\infty}e^{-ch^2}dh\\
    &\leq \frac{C}{\sqrt{x+1}},
\end{align*}
where we have used \eqref{fu} and $d\psi=\sqrt{x+1}dh.$ Then our result follows from Theorem \ref{thm:w-decay} and Proposition \ref{52}.
\end{proof}

\bigbreak

\section*{Acknowledgments}
Y. Wang is  supported by NSFC under Grant 12001383. Z. Zhang is partially supported by  NSF of China  under Grant 12171010.

\end{document}